\documentclass[10pt,leqno]{amsart}
\usepackage{amssymb,amsthm}
\usepackage{amsmath}
\usepackage{setspace}
\usepackage{dcpic,pictexwd}
\usepackage[all]{xy}
\usepackage[pdftex]{graphicx}
\usepackage{mathrsfs}
\usepackage[pdftex]{hyperref}
\usepackage{bbm}

\theoremstyle{definition}
 \newtheorem{definition}{Definition}[section]

\theoremstyle{plain}
 \newtheorem{proposition}[definition]{Proposition}

\theoremstyle{plain}
 \newtheorem{theorem}[definition]{Theorem}
 \newtheorem{claim}[definition]{Claim}
\theoremstyle{definition}
 \newtheorem{example}[definition]{Example}

\theoremstyle{plain}
 \newtheorem{lemma}[definition]{Lemma}

\theoremstyle{plain}
 \newtheorem{corollary}[definition]{Corollary}

\theoremstyle{remark}
 \newtheorem{remark}[definition]{Remark}

\theoremstyle{definition}

\theoremstyle{plain}

\newcommand{\Ext}{\mathrm{Ext}}
\newcommand{\End}{\mathrm{End}}

\newcommand{\Hom}{\mathrm{Hom}}
\newcommand{\SHom}{\underline{\mathrm{Hom}}}

\newcommand{\Ca}{\mathcal{C}}
\newcommand{\Fun}{\mathrm{F}}

\newcommand{\Def}{\mathrm{Def}}

\newcommand{\Ob}{\mathrm{Ob}}

\newcommand{\Z}{\mathbb{Z}}
\newcommand{\SEnd}{\underline{\End}}
\newcommand{\A}{\Lambda}

\newcommand{\m}{\mathfrak{m}}
\renewcommand{\k}{\Bbbk}

\newcommand{\Size}{500}
\newcommand{\Sized}{300}

\hypersetup{
    bookmarks=true,         % show bookmarks bar?
    unicode=false,          % non-Latin characters in AcrobatÕs bookmarks
    pdftoolbar=true,        % show AcrobatÕs toolbar?
    pdfmenubar=true,        % show AcrobatÕs menu?
    pdffitwindow=false,     % window fit to page when opened
    pdfstartview={FitH},    % fits the width of the page to the window
    pdftitle={Universal Deformation Rings for Complexes Over Finite-Dimensional Algebras},    % title
    pdfauthor={Jose A. Velez-Marulanda},     % author
    pdfsubject={Derived Categories, Singularity Categories, Universal deformation rings},   % subject of the document
    pdfcreator={Jose A. Velez-Marulanda},   % creator of the document
    pdfproducer={Jose A. Velez-Marulanda}, % producer of the document
    pdfkeywords={Derived Categories} {Singularity Categories}{Universal Deformation Rings}, % list of keywords
    pdfnewwindow=true,      % links in new window
    colorlinks=true,       % false: boxed links; true: colored links
    linkcolor=red,          % color of internal links
    citecolor=blue,        % color of links to bibliography
    filecolor=magenta,      % color of file links
    urlcolor=blue           % color of external links
}

\title[Universal Deformation Rings for Complexes]{Universal Deformation Rings for Complexes over Finite-Dimensional Algebras}
\author{Jos\'e A. V\'elez-Marulanda}
\address{Department of Mathematics, Valdosta State University,
2072 Nevins Hall, 1500 N. Patterson St, Valdosta, GA,  31698-0040}
\email{javelezmarulanda AT valdosta DOT edu}


\begin{document}
\renewcommand{\labelenumi}{\textup{(\roman{enumi})}}
\renewcommand{\labelenumii}{\textup{(\roman{enumi}.\alph{enumii})}}
\numberwithin{equation}{section}

%\halfspacing

\begin{abstract}
Let $\k$ be field of arbitrary characteristic and let $\A$ be a finite dimensional $\k$-algebra. From results previously obtained by F.M Bleher and the author, it follows that if $V^\bullet$ is an object of the bounded derived category $\mathcal{D}^b(\A\textup{-mod})$ of $\A$, then $V^\bullet$ has a well-defined versal deformation ring $R(\A, V^\bullet)$, which is complete local commutative Noetherian $\k$-algebra with residue field $\k$, and which is universal provided that $\Hom_{\mathcal{D}^b(\A\textup{-mod})}(V^\bullet, V^\bullet)=\k$. Let $\mathcal{D}_\textup{sg}(\A\textup{-mod})$ denote the singularity category of $\A$ and assume that $V^\bullet$ is a bounded complex whose terms are all finitely generated Gorenstein projective left $\A$-modules. In this article we prove that if $\Hom_{\mathcal{D}_\textup{sg}(\A\textup{-mod})}(V^\bullet, V^\bullet)=\k$, then the versal deformation ring $R(\A, V^\bullet)$ is universal. We also prove that certain singular equivalences of Morita type (as introduced by X. W. Chen and L. G. Sun) preserve the isomorphism class of versal deformation rings of bounded complexes whose terms are finitely generated Gorenstein projective $\A$-modules.
\end{abstract}
%\dedicatory{To my good friends Wildomar Alarc\'on and Ricardo Pe{\~n}a for their unconditional support.}
%\thanks{The second author was supported by the Release Time for Research Scholarship of the Office of Academic Affairs  at the Valdosta State University.}
\maketitle
%\tableofcontents

\section{Introduction}

Let $\k$ be a field of arbitrary characteristic, let $V$ be a $\A$-module that has finite dimension over $\k$ and let $R$ be a complete local commutative Noetherian $\k$-algebra with residue field $\k$. A lift of $V$ over $R$ is an $R\otimes_\k\A$-module $M$ that is free over $R$ together with an $\A$-module isomorphism $\phi: \k\otimes_RM\to V$.   A deformation of $V$ over $R$ is defined to be an isomorphism class of lifts of $V$ over $R$. In \cite{blehervelez}, F. M. Bleher and the author to study universal deformations rings and deformations of modules for arbitrary finite dimensional $\k$ algebras. In particular, they proved that when $\A$ is a self-injective algebra and $V$ is a $\A$-module with finite dimension over $\k$ such that the stable endomorphism ring of $V$ is trivial, then $V$ has an universal deformation ring $R(\A,V)$ that is stable under taking syzygies. This approach was used  in \cite{bleher9,blehervelez,velez}, to study universal deformation rings for certain self-injective algebras which are not Morita equivalent to a block of a group algebra.  In \cite{velez2}, the results in \cite{blehervelez} were extended for the case when $\A$ is a Gorenstein $\k$-algebra (as introduced in \cite{auslander2}), and for when $V$ is a (maximal) Cohen-Macaulay $\A$-module (as introduced in \cite{buchweitz}). More recently, F. M. Bleher and the author in \cite{blehervelez2} extended some of the results in \cite{blehervelez} to bounded complexes over finite dimensional algebras by adapting the techniques in \cite{bleher13,bleher14}. More precisely, they proved that if $V^\bullet$ is a bounded complex of finitely generated $\A$-modules, then $V^\bullet$ has a well-defined versal deformation ring $R(\A,V^\bullet)$, which is universal provided that the endomorphism ring of $V^\bullet$ over the derived category of $\A$ is trivial. In \cite{velez2}, the results in \cite{blehervelez} were extended for the case when $\A$ is a Gorenstein $\k$-algebra (as introduced in \cite{auslander2}), and for when $V$ is a (maximal) Cohen-Macaulay $\A$-module (as introduced in \cite{buchweitz}), i.e., $\Ext_\A^i(V,\A)=0$ for all $i>0$. Following \cite{enochs0,enochs}, we say that a (not necessarily finitely generated) left $\A$-module $M$ is {\it Gorenstein projective} provided that there exists an exact sequence of (not necessarily finitely generated) projective left $\A$-modules 
\begin{equation*}
\cdots\to P^{-2}\xrightarrow{f^{-2}} P^{-1}\xrightarrow{f^{-1}} P^0\xrightarrow{f^0}P^1\xrightarrow{f^1}P^2\to\cdots
\end{equation*}  
such that $M=\mathrm{coker}\,f^0$, and for all integers $i>0$ and $j\in \Z$ we have $\Ext_\A^i(\ker\,f^j, \A)=0$. 
Following \cite{auslander4} and \cite{avramov2}, we say that a finitely generated left $\A$-module $M$ is of {\it Gorentein dimension zero} or {\it totally reflexive} provided that the left $\A$-modules $M$ and $\Hom_\A(\Hom_\A(M,\A),\A)$ are isomorphic, and that $\Ext_\A^i(M,\A)=0=\Ext_\A^i(\Hom_\A(M,\A),\A)$ for all $i>0$.  It is well-known that finitely generated Gorenstein projective left $\A$-modules coincided with those that are totally reflexive (see e.g. \cite[Lemma 2.1.4]{chenxw4}). Moreover, if $\A$ is a Gorenstein $\k$-algebra, then finitely generated Gorenstein projective left $\A$-modules and (maximal) Cohen-Macaulay $\A$-modules coincide (see \cite{auslander3}). However, it follows from an example given by J.I. Miyachi (see \cite[Example A.3]{hoshino-koga}) that in general not all (maximal) Cohen-Macaulay modules over a finite dimensional algebra are Gorenstein projective.
Recall that the singularity category $\mathcal{D}_{\mathrm{sg}}(\A\textup{-mod})$ of $\A$ is the Verdier quotient of the bounded derived 
category of finitely generated left $\A$-modules $\mathcal{D}^b(\A\textup{-mod})$ by the full subcategory $\mathcal{K}^b(\A\textup{-proj})$ of perfect complexes (see 
\cite{verdier}, \cite{krause3} and \S \ref{sec2} for more details). If $\A$ is self-injective, then it follows from \cite[Thm. 2.1]
{rickard1} that $\mathcal{D}_{\mathrm{sg}}(\A\textup{-mod})$ is equivalent as a triangulated category to $\A\textup{-\underline{mod}}$ the 
stable category of finitely generated left $\A$-modules. If $\A$ is Gorenstein, then it follows from \cite{buchweitz} (see also \cite[\S4.6]{happel3} for the case when $\k$ is algebraically closed) that $\mathcal{D}_{\mathrm{sg}}(\A\textup{-mod})$ is equivalent as a triangulated 
category to the stable category of finitely generated Gorenstein projective left $\A$-modules. In particular, if $\A$ has finite global dimension, then its singularity category is trivial.  In a more general setting, D. Orlov rediscovered independently the notion of singularity category in the context of algebraic geometry and 
mathematical physics (see e.g. \cite{orlov1,orlov2}).  This has motivated the study of the structure of singularity categories under different situations (see e.g. \cite{chenlu,chenxw1,chenxw2,chenye,kalck,shen} and their references).

The first goal of this article is to prove the following result.

\begin{theorem}\label{thm01}
Let $\A$ be a finite dimensional $\k$-algebra, and let $V^\bullet$ be a {\bf bounded} complex in $\mathcal{D}^b(\A\textup{-mod})$ such that all of its terms are Gorenstein projective left $\A$-modules. 
\begin{enumerate}
\item If $\Hom_{\mathcal{D}_\textup{sg}(\A\textup{-mod})}(V^\bullet, V^\bullet)=\k$, then the versal deformation ring $R(\A,V^\bullet)$ is universal. 
\item For all perfect complexes $P^\bullet$ over $\A$, the versal deformation ring $R(\A, V^\bullet\oplus P^\bullet)$ is isomorphic to $R(\A,V^\bullet)$.
\end{enumerate} 
\end{theorem}

Examples of complexes $V^\bullet$ as in the hypothesis of Theorem \ref{thm01} can be found in the study of $\mathcal{G}$-resolutions of modules of finite Gorenstein dimension as explained in \cite{avramov2}, in the study of Gorenstein derived categories as introduced in \cite{gaozhang}, and in the study of Gorenstein singularity categories as introduced in \cite{baoduzhao}. 

It was also proved in \cite{blehervelez2} that versal deformation rings of complexes are preserved by derived equivalences induced by so-called nice two-sided tilting complexes. Moreover, it was proved in \cite{blehervelez2} that the isomorphism class of versal deformation rings of modules is preserved by stable equivalences of Morita type (as introduced by M. Brou\'e in \cite{broue}) between self-injective $\k$-algebras. In \cite{bekkert-giraldo-velez}, V. Bekkert, H. Giraldo and the author proved that the isomorphism class of versal deformation rings of (maximal) Cohen-Macaulay modules over Gorenstein algebras is also preserved by so-called {\it singular equivalences of Morita type}, which extends the aformentioned result in \cite{blehervelez2}. These singular equivalences of Morita type were introduced by X. W. Chen and L. G. Sun in an unpublished article (see \cite{chensun}) and then formally discussed by G. Zhou and A. Zimmermann in \cite{zhouzimm} as a way of generalizing the concept of stable equivalence of Morita type to singularity categories. 

The second goal of this article consists in proving the following result (for more details see Definition \ref{defi:3.2} and Theorem \ref{thm:3.6}).

\begin{theorem}\label{thm02}
Let $\A$ and $\Gamma$ be finite dimensional $\k$-algebras, and assume that ${_\Gamma}X_\A$, ${_\A} Y_\Gamma$ are bimodules that induce a singular equivalence of Morita type in the sense of \cite{chensun,zhouzimm} and Definition \ref{defi:3.2}, such that $\Hom_\Gamma(X,\Gamma)$ and $\Hom_\A(Y,\A)$ are of finite projective dimension as a left $\A$-module and as a left $\Gamma$-module, respectively. Let $V^\bullet$ be a bounded complex in $\mathcal{D}^b(\A\textup{-mod})$ whose terms are all finitely generated Gorenstein projective left $\A$-modules. Then the terms of $X\otimes_\A V^\bullet$ are all finitely generated Gorenstein projective left  $\Gamma$-modules, and the versal deformation rings $R(\A, V^\bullet)$ and $R(\Gamma, X\otimes_\A V^\bullet)$ are isomorphic in $\hat{\Ca}$. 
\end{theorem}

%Recall that (maximal) Cohen-Macaulay modules coincide with the ones with finite Gorenstein dimension zero for the case when the underlying algebra is Gorenstein (see e.g. \cite{auslander3,buchweitz}). However, it follows from an example given by J.I. Miyachi (see \cite[Example A.3]{hoshino-koga}) that in general not all (maximal) Cohen-Macaulay modules over a finite dimensional algebra are of Gorenstein dimension zero. 

This article is organized as follows. In \S \ref{sec2}, we fix the notation that will be used in this article, review some preliminary definitions and properties concerning perfect complexes over finite dimensional algebras, and provide some basic result that will be used in the following sections. We also review some definitions and important properties concerning deformations and (uni)versal deformation rings of complexes over finite dimensional algebras as discussed in \cite{blehervelez2}. In \S \ref{section3} we prove Theorem \ref{thm01} by carefully adapting some of the ideas in the proof of \cite[Thm. 2.6]{blehervelez} to the context of derived categories. In \S\ref{section4}, we prove Theorem \ref{thm02} by adapting the ideas in the proofs of \cite[Prop. 3.2.6]{blehervelez2} and \cite[Thm. 3.4]{bekkert-giraldo-velez} to our context. As a consequence, we obtain that universal deformation rings of complexes whose terms are finitely generated modules of Gorenstein dimension zero and whose endomorphism ring (as an objected of the singularity category) is trivial, are invariant under certain singular equivalences of Morita type. 

%Finally, in \S \ref{section5} we discuss an example of a non-Gorenstein finite dimensional $\k$-algebra $\A_0$, which has exactly two isomorphism classes of modules of Gorenstein dimension zero.  If $V_1$ and $V_2$ are $\A_0$-modules representing such isomorphism classes, then we prove that $\Omega V_1=V_2$, $\Hom_{\mathcal{D}_\textup{sg}(\A_0\textup{-mod})}(V_i,V_i)=\k$ for $i=1,2$, and that the corresponding universal deformation rings $R(\A,V_1)$ and $R(\A,V_2$ are both isomorphic to $\k[[t]]/(t^3)$.
We refer the reader to look at \cite{auslander4,auslander3,avramov2,beligiannis1,beligiannis3,beligiannis2,buchweitz,enochs0,enochs} (and their references) for basic concepts concerning (maximal) Cohen-Macaulay, totally reflexive and Gorenstein projective modules as well as their applications in other settings, and refer to \cite{krause3,verdier} for getting detailed information concerning localizations of triangulated categories and Verdier quotients.

\section{Notations, preliminaries and basic results}\label{sec2}
In this article, $\k$ denotes a field of arbitrary characteristic and $\A$ denotes a fixed but arbitrary finite-dimensional $\k$-algebra.  We denote by $\hat{\Ca}$ the category of complete local commutative Noetherian $\k$-algebras with residue field $\k$. In particular, the morphisms in $\hat{\Ca}$ are continuous $\k$-algebra homomorphisms which induce the identity on $\k$. We denote by $\Ca$ the full subcategory of Artinian objects in $\hat{\Ca}$. For every Noetherian object $R$ in $\hat{\Ca}$, we denote by $R\A$ the tensor product of $\k$-algebras $R\otimes_\k\A$,  and by $R\A\textup{-mod}$ we denote the abelian category of all finitely generated left $R\A$-modules. We denote by $R\A$-proj the full subcategory of $R\A$-mod whose objects are finitely generated projective $R\A$-modules. 

Let $R$ be a fixed Noetherian object in $\hat{\Ca}$. Following \cite{brummer, gabriel1, gabriel2}, an $R\A$-module $M$ is said to be {\it pseudocompact} if it is the inverse limit of $R\A$-modules of finite length having the discrete topology. In particular, every finitely generated $R\A$-module is pseudocompact. It follows from \cite{gabriel1,gabriel2} that the category of pseudocompact left $R\A$-modules is an abelian category. Let $\mathcal{C}^-(R\A)$ denote the abelian category of bounded above complexes of pseudocompact $R\A$-modules, let $\mathcal{K}^-(R\A)$ be the corresponding homotopy category, and let $\mathcal{D}^-(R\A)$ be the corresponding derived category. In particular, the bounded derived category of finitely generated left $R\A$-modules $\mathcal{D}^b(R\A\textup{-mod})$ is a full triangulated subcategory of $\mathcal{D}^-(R\A)$. If $M^\bullet$ is an object in $\mathcal{C}^-(R\A)$, then we assume that $M^\bullet$ is of the form 
\begin{equation*}
M^\bullet: \cdots\to M^n\xrightarrow{\delta_M^n}M^{n+1}\xrightarrow{\delta_M^{n+1}}M^{n+2}\to \cdots\to M^{n+m-1}\xrightarrow{\delta^{n+m-1}_M} M^{n+m}\to 0\to \cdots,
\end{equation*} 
and that $\delta_M^{i+1}\circ \delta_M^i=0$ for $n\leq i\leq n+m-1$. We denote by $T$ the shifting functor on $\mathcal{C}^-(R\A)$, $\mathcal{K}^-(R\A)$ and $\mathcal{D}^-(R\A)$, i.e. $T$ shifts complexes to the left and changes the sign of the differentials. If $M$ is a pseudocompact $R\A$-module, then $M$ defines an object in $\mathcal{D}^-(R\A)$ by considering it as a stalk complex concentrated in degree zero. For all morphisms $f:M^\bullet\to M'^\bullet$ in $\mathcal{C}^-(R\A)$, we denote by $C^\bullet_f$ the mapping cone of $f$ in $\mathcal{K}^-(R\A)$ (see e.g. \cite[\S I. 2]{hartshorne}), and denote also by $f$ the corresponding chain homotopy class in $\mathcal{K}^-(R\A)$.

We denote by $\mathcal{K}^b(R\A\textup{-proj})$ the full triangulated subcategory of $\mathcal{D}^-(R\A)$ consisting in those objects that are bounded complexes of finitely generated projective $R\A$-modules. We say that an object $P^\bullet$ of $\mathcal{D}^-(R\A)$ is a {\it perfect complex over $R\A$} if it is isomorphic to an object of $\mathcal{K}^b(R\A\textup{-proj})$. By \cite[Lemma 1.2.1]{buchweitz}, $\mathcal{K}^b(R\A\textup{-proj})$ is a thick (or \'epaisse) subcategory of $\mathcal{D}^b(R\A\textup{-mod})$ in the sense of \cite[\S I.2.1.1]{verdier}. Since by \cite[Prop. 1.3]{rickard1}, $\mathcal{D}^b(R\A\textup{-mod})$ is a thick subcategory of $\mathcal{D}^-(R\A)$, it follows from \cite[Cor. 4.3]{verdier} that $\mathcal{K}^b(R\A\textup{-proj})$ is also a thick subcategory of $\mathcal{D}^-(R\A)$.  We denote by $\mathcal{D}^-_{\mathfrak{perf}}(R\A)$ the Verdier quotient 
\begin{equation*}
\mathcal{D}^-_{\mathfrak{perf}}(R\A)=\mathcal{D}^-(R\A)/\mathcal{K}^b(R\A\textup{-proj}).
\end{equation*}
More precisely, $\mathcal{D}^-_{\mathfrak{perf}}(R\A)$ is the localization $\mathcal{D}^-(R\A)[\Sigma(\mathcal{K}^b(R\A\textup{-proj}))^{-1}]$ of $\mathcal{D}^-(R\A)$ by the multiplicative system $\Sigma(\mathcal{K}^b(R\A\textup{-proj}))$ consisting in those morphisms $X^\bullet \to Y^\bullet$ in $\mathcal{D}^-(R\A)$ that fit into an exact triangle $X^\bullet\to Y^\bullet \to Z^\bullet \to TX^\bullet$  with $Z^\bullet$ isomorphic to a perfect complex over $R\A$. In particular, a morphism $u:M^\bullet\to M'^\bullet$ in $\mathcal{C}^-(R\A)$ belongs to $\Sigma(\mathcal{K}^b(R\A\textup{-proj}))$ if and only if $C^\bullet_u$ is isomorphic to an object of $\mathcal{K}^b(R\A\textup{-proj})$. A morphism $\hat{f}: M^\bullet \to M'^\bullet$ in $\mathcal{D}^-_{\mathfrak{perf}}(R\A)$ can be represented by a diagram of the form
\begin{equation*}
\begindc{\commdiag}[\Size]
\obj(1,-1)[p0]{$M''^\bullet$}
\obj(0,0)[p1]{$M^\bullet$}
\obj(2,0)[p2]{$M'^\bullet$}
\mor{p1}{p0}{$f$}[-1,0]
\mor{p2}{p0}{$u$}
\enddc
\end{equation*} 
where $f$ is a morphism in $\mathcal{D}^-(R\A)$ and $u$ is a morphism in $\Sigma(\mathcal{K}^b(R\A\textup{-proj}))$.  On the other hand, a morphism $f^\bullet: M^\bullet \to M'^\bullet$ in $\mathcal{D}^-(R\A)$ becomes zero in $\mathcal{D}^-_{\mathfrak{perf}}(R\A)$ if and only if there exists a morphism $u$ in $\Sigma(\mathcal{K}^b(R\A\textup{-proj}))$ such that $u\circ f=0$ if and only if there exists a morphism $v$ in $\Sigma(\mathcal{K}^b(R\A\textup{-proj}))$ such that $f\circ v=0$ if and only if $f$ factors through some object isomorphic to a perfect complex over $R\A$ (see  e.g. \cite[\S 3 \& \S 4]{krause3} for more details). Following \cite{orlov1,orlov2}, we denote by  $\mathcal{D}_\textup{sg}(R\A\textup{-mod})$ the singularity category of $R\A$, i.e., the Verdier quotient
\begin{equation*}
\mathcal{D}_\textup{sg}(R\A\textup{-mod})=\mathcal{D}^b(R\A\textup{-mod})/\mathcal{K}^b(R\A\textup{-proj}).
\end{equation*}
Note in particular that by \cite[Cor. 4.3]{verdier}, $\mathcal{D}_\textup{sg}(R\A\textup{-mod})$ is a thick subcategory of $\mathcal{D}^-_{\mathfrak{perf}}(R\A)$.

Following \cite[Def. 5.2]{illusie1} (see also \cite[\S 8.3.6]{illusie2}), we say that an object $M^\bullet$ in $\mathcal{K}^-(R\A)$ has {\it finite pseudocompact $R$-tor dimension}, provided that there exists an integer $N$ such that for all pseudocompact $R$-modules $S$, and for all integers $i<N$ we have $\mathrm{H}^i(S\hat{\otimes}_R^\mathbf{L}M^\bullet)=0$, where $\hat{\otimes}_R$ denotes the completed tensor product in the category of pseudocompact $R$-modules. Note in particular that if $M$ is a finitely generated as pseudocompact $R$-module, then the functors $M\otimes_R-$ and $M\hat{\otimes}_R-$ are naturally isomorphic. From now on we use $\hat{\otimes}_R$ to denote both $\otimes_R$ and $\hat{\otimes}_R$ accordingly. If $\pi:R\to R'$ is a morphism of Noetherian objects in $\hat{\Ca}$, then $\pi$ induces a morphism of derived categories $R'\hat{\otimes}^\mathbf{L}_{R,\pi}-:\mathcal{D}^-(R\A)\to \mathcal{D}^-(R'\A)$. In particular, for all objects $M^\bullet$ in $\mathcal{D}^-(R\A)$, $R'\hat{\otimes}^\mathbf{L}_{R,\pi}TM^\bullet=T(R'\hat{\otimes}^\mathbf{L}_{R,\pi}M^\bullet)$. 

\begin{remark}\label{remark:1.1}
Let $R$ be an Artinian object in $\Ca$ and suppose that $M^\bullet$ is a complex in $\mathcal{K}^-(R\A)$ that has abstractly free pseudocompact $R$-modules and has finite pseudocompact $R$-tor dimension. By the arguments in \cite[Remark 2.1.6]{blehervelez2} (see also \cite[pg. 263]{milne}), we can truncate $M^\bullet$ to obtain a bounded complex  $M'^\bullet$ whose terms are  abstractly free pseudocompact $R$-modules and such that $M^\bullet$ is quasi-isomorphic to $M'^\bullet$ in $\mathcal{C}^-(R\A)$.
\end{remark}
The following lemma is a direct consequence of a result due to M. Hashimoto (see \cite[Thm. 1]{hashimoto}).

\begin{lemma}\label{lemma:1.0}
Let $R$ be a Noetherian object in $\hat{\Ca}$ and let $M^\bullet$ be a complex in $\mathcal{C}^-(R\A)$ whose terms are finitely generated flat $R$-modules. If $\k\hat{\otimes}_RM^\bullet$ is acyclic, then $M^\bullet$ is also acyclic. 
\end{lemma}

We obtain the following result, which relates surjective morphisms of Artinian objects in $\Ca$ with perfect and acyclic complexes.

\begin{lemma}\label{lemma:1.1}
Let $\pi: R\to R_0$ be a surjective morphism of Artinian objects in $\Ca$. 
\begin{enumerate}
\item If $P_0^\bullet$ is an object in $\mathcal{K}^b(R_0\A\textup{-proj})$, then there exists an object $P^\bullet$ in $\mathcal{K}^b(R\A\textup{-proj})$ such that $R_0\hat{\otimes}_{R,\pi}P^\bullet$ is isomorphic to $P_0^\bullet$ in $\mathcal{C}^-(R_0\A)$. 
\item Let $f:M^\bullet \to M'^\bullet$ be a morphism in $\mathcal{C}^-(R\A)$ where the terms of both $M^\bullet$ and $M'^\bullet$ are abstractly free finitely generated $R\A$-modules. If the induced morphism $\k\hat{\otimes}_Rf: \k\hat{\otimes}_RM^\bullet \to \k\hat{\otimes}_RM'^\bullet$ is a quasi-isomorphism in $\mathcal{C}^-(\A)$, then $f$ is also a quasi-isomorphism in $\mathcal{C}^-(R\A)$.
\item Let $u: M^\bullet\to M''^\bullet$, $v: M^\bullet \to M'^\bullet$ and $f: M'^\bullet \to M''^\bullet$ be morphisms in $\mathcal{C}^-(R\A)$ such that $f\circ v=u$. If the mapping cones $C_u^\bullet$ and $C_v^\bullet$ of $u$ and $v$, respectively, are both perfect complexes over $R\A$, then so is $C_f^\bullet$ the mapping cone of $f$.  
\end{enumerate}
\end{lemma}
\begin{proof}
(i). We can assume without loosing generality that $P_0^\bullet$ is of the form 
\begin{equation*}
P_0^\bullet: 0\to P_0^{-m}\xrightarrow{\delta_{P_0}^{-m}}P_0^{-m+1}\xrightarrow{\delta_{P_0}^{-m+1}}P_0^{-m+2}\to \cdots\to P_0^{-1}\xrightarrow{\delta_{P_0}^{-1}} P_0^0\to 0,
\end{equation*}
where for all $0\leq j\leq m$, $P_0^{-j}$ is a finitely generated projective $R_0\A$-module. For each $P_0^{-j}$, there exists a finitely generated projective $R\A$-module $P^{-j}$ and a surjective morphism of $R\A$-modules $p^{-j}: P^{-j}\to P_0^{-j}$ such that $R_0\hat{\otimes}_{R, \pi} p^{-j}: R_0\hat{\otimes}_{R,\pi}P^{-j}\to P_0^{-j}$ is an isomorphism of finitely generated projective $R_0\A$-modules. Moreover, for all $0\leq j \leq m$, there exists a morphism $\delta_{P}^{-j}: P^{-j}\to P^{-j+1}$ with $\delta_P^0=0$ such that $p^{-j+1}\circ \delta_P^{-j}=\delta_{P_0}^{-j}\circ p^{-j}$ as $R\A$-module homomorphisms.  Therefore 
\begin{align*}
\k\hat{\otimes}_R(\delta_P^{-j+1}\circ \delta_P^{-j})&=\k\hat{\otimes}_{R_0}(R_0\hat{\otimes}_{R,\pi}(\delta_P^{-j+1}\circ \delta_P^{-j}))=\k\hat{\otimes}_{R_0}(\delta_{P_0}^{-j+1}\circ \delta_{P_0}^{-j})=0.
%&=\k\hat{\otimes}_{R_0}(R_0\hat{\otimes}_{R,\pi}\delta_P^{-j+1}\circ R_0\hat{\otimes}_{R,\pi}\delta_P^{-j})\\
\end{align*}
It follows by Nakayama's Lemma that $\delta_P^{-j+1}\circ\delta_P^j=0$ for all $0\leq j\leq m$. Thus we obtain a perfect complex over $R\A$, namely  
\begin{equation*}
P^\bullet: 0\to P^{-m}\xrightarrow{\delta_{P}^{-m}}P^{-m+1}\xrightarrow{\delta_P^{-m+1}}P^{-m+2}\to \cdots\to P^{-1}\xrightarrow{\delta_P^{-1}} P^0\to 0,
\end{equation*}
such that $R_0\otimes_{R,\pi}P^\bullet$ is isomorphic to $P_0^\bullet$ in $\mathcal{K}^b(R_0\A\textup{-mod})$.

(ii). Let $C_f^\bullet$ be the mapping cone of $f$ in $\mathcal{K}^-(R\A)$. After tensoring the corresponding distinguished triangle $M^\bullet\xrightarrow{f} M'^\bullet \to C_f^\bullet \to TM^\bullet$ in $\mathcal{K}^-(R\A)$ with $\k$ over $R$ and using  the axiom of triangulated categories (TR2) (see e.g. \cite[\S I.1]{hartshorne}) together with \cite[\S Prop. I.1.1 (c)]{hartshorne}, we obtain an isomorphism $h: \k\hat{\otimes}_{R,\pi} C_f^\bullet\to C_{\k\hat{\otimes}_Rf}^\bullet$ such that the following diagram of triangles in $\mathcal{K}^-(\A)$ is commutative: 
\begin{equation*}
\begindc{\commdiag}[\Sized]
\obj(0,2)[p0]{$\k\hat{\otimes}_RM^\bullet$}
\obj(4,2)[p1]{$\k\hat{\otimes}_RM'^\bullet$}
\obj(8,2)[p2]{$\k\hat{\otimes}_RC_f^\bullet$}
\obj(12,2)[p3]{$\k\hat{\otimes}_RTM^\bullet$}
\obj(0,0)[q0]{$\k\hat{\otimes}_RM^\bullet$}
\obj(4,0)[q1]{$\k\hat{\otimes}_RM'^\bullet$}
\obj(8,0)[q2]{$C_{\k\hat{\otimes}_Rf}^\bullet$}
\obj(12,0)[q3]{$T(\k\hat{\otimes}_RM^\bullet)$}
\mor{p0}{p1}{$\k\hat{\otimes}_Rf$}
\mor{p1}{p2}{}
\mor{p2}{p3}{}
\mor{q0}{q1}{$\k\hat{\otimes}_Rf$}
\mor{q1}{q2}{}
\mor{q2}{q3}{}
\mor{p0}{q0}{$=$}
\mor{p1}{q1}{$=$}
\mor{p2}{q2}{$h$}
\mor{p3}{q3}{$=$}
\enddc
\end{equation*}   
Since $\k\hat{\otimes}_R f$ is a quasi-isomorphism in $\mathcal{C}^-(\A)$, it follows from \cite[Cor. 1.5.4]{weibel} that $C_{\k\hat{\otimes}_Rf}^\bullet$ is acyclic, which implies that $\k\hat{\otimes}_RC_f^\bullet$ is also acyclic. Note in particular that the terms of $C_f^\bullet$ are free (so flat) as $R$-modules. By Lemma \ref{lemma:1.0}, it follows that $C_f^\bullet$ is acyclic and consequently $f$ is a quasi-isomorphism in $\mathcal{C}^-(R\A)$.

(iii). After using  the Octahedral Axiom (TR4) of triangulated categories (see e.g. \cite[\S I.1]{hartshorne} or \cite[\S 4.1]{krause3}), we obtain the following commutative diagram of triangles in $\mathcal{K}^-(R\A)$:
\begin{equation*}
\begindc{\commdiag}[\Sized]
\obj(0,2)[p0]{$M^\bullet$}
\obj(4,2)[p1]{$M'^\bullet$}
\obj(8,2)[p2]{$C_v^\bullet$}
\obj(12,2)[p3]{$TM^\bullet$}
\obj(0,0)[q0]{$M^\bullet$}
\obj(4,0)[q1]{$M''^\bullet$}
\obj(8,0)[q2]{$C_u^\bullet$}
\obj(12,0)[q3]{$TM^\bullet$}
\obj(4,-2)[r1]{$C_f^\bullet$}
\obj(8,-2)[r2]{$C_f^\bullet$}
\obj(12,-2)[r3]{$TM'^\bullet$}
\obj(4,-4)[s1]{$T M'^\bullet$}
\obj(8,-4)[s2]{$TC_v^\bullet$}
\mor{p0}{p1}{$v$}
\mor{p1}{p2}{}
\mor{p2}{p3}{}
\mor{q0}{q1}{$u$}
\mor{q1}{q2}{}
\mor{q2}{q3}{}
\mor{p0}{q0}{$=$}
\mor{p1}{q1}{$f$}
\mor{p2}{q2}{$h$}
\mor{p3}{q3}{$=$}
\mor{q1}{r1}{}
\mor{r1}{s1}{}
\mor{r1}{r2}{$=$}
\mor{r2}{r3}{}
\mor{s1}{s2}{}
\mor{q2}{r2}{}
\mor{r2}{s2}{}
\mor{q3}{r3}{$Tv$}
\enddc
\end{equation*}   
Since the category of perfect complexes over $R\A$ is a full triangulated subcategory of $\mathcal{K}^-(R\A)$ , it follows that $C_f^\bullet$ is also a perfect complex over $R\A$. This finishes the proof of Lemma \ref{lemma:1.1}.
\end{proof}

Following \cite{enochs0,enochs}, we say that a (not necessarily finitely generated) left $\A$-module $M$ is {\it Gorenstein projective} provided that there exists an exact sequence of (not necessarily finitely generated) projective left $\A$-modules 
\begin{equation*}
\cdots\to P^{-2}\xrightarrow{f^{-2}} P^{-1}\xrightarrow{f^{-1}} P^0\xrightarrow{f^0}P^1\xrightarrow{f^1}P^2\to\cdots
\end{equation*}  
such that $M=\mathrm{coker}\,f^0$, and for all integers $i>0$ and $j\in \Z$ we have $\Ext_\A^i(\ker\,f^j, \A)=0$. 
Following \cite{auslander4} and \cite{avramov2}, we say that a finitely generated left $\A$-module $M$ is of {\it Gorenstein dimension zero} or {\it totally reflexive} provided that the left $\A$-modules $M$ and $\Hom_\A(\Hom_\A(M,\A),\A)$ are isomorphic, and that $\Ext_\A^i(M,\A)=0=\Ext_\A^i(\Hom_\A(M,\A),\A)$ for all $i>0$.  It is well-known that finitely generated Gorenstein projective left $\A$-modules coincided with those that are totally reflexive (see e.g. \cite[Lemma 2.1.4]{chenxw4}). We denote by $\A\textup{-GProj}$ the category of Gorenstein projective left $\A$-modules.
%, and by $\A\textup{-PCGProj}$ the category of all pseudocompact Gorenstein projective left $\A$-modules. 

Note that if $V$ is in $\A\textup{-GProj}$ and $P$ is a finitely generated projective $\A$-module, then $\Ext_{\A}^i(V,P)=0$ for all $i>0$. 

\begin{lemma}\label{lemma:1.3}
Let $P^\bullet$ be an arbitrary object in $\mathcal{K}^b(\A\textup{-proj})$.
 \begin{enumerate}
\item If $V\in \A\textup{-GProj}$, then $\Ext_{\mathcal{D}^-(\A)}^i(V, P^\bullet)=\Hom_{\mathcal{D}^-(\A)}(V, T^iP^\bullet)=0$ for all $i>0$.
\item If $V^\bullet$ is a {\bf bounded} complex in $\mathcal{C}^-(\A)$ with all of its terms in $\A\textup{-GProj}$, then for all $i>0$
\begin{equation*}
\Ext_{\mathcal{D}^-(\A)}^i(V^\bullet, P^\bullet)=\Hom_{\mathcal{D}^-(\A)}(V^\bullet, T^iP^\bullet)=0.
\end{equation*}
\end{enumerate}
\end{lemma}
\begin{proof}
(i). Assume without loosing generality that $P^\bullet$ is of the form:
\begin{equation*}
P^\bullet: 0\to P^{-m}\xrightarrow{\delta_{P}^{-m}}P^{-m+1}\xrightarrow{\delta_P^{-m+1}}P^{-m+2}\to \cdots\to P^{-1}\xrightarrow{\delta_P^{-1}} P^0\to 0, 
\end{equation*}
where for all $0\leq j\leq m$, $P^{-j}$ is a finitely generated projective $\A$-module.
Since $V\in \A\textup{-GProj}$, Lemma \ref{lemma:1.3} (i) follows for the case $m=0$. Assume the statement of Lemma \ref{lemma:1.3} (i) true for all perfect complexes $Q^\bullet$ over $\A$ whose length is smaller than $m$.  Let $P'^\bullet$ be the associated truncated complex to $P^\bullet$ (i.e. $P'^{-j}=0$ for $j\geq m$, $P'^{-j}=P^{-j}$ and $\delta_{P'}^{-j} = \delta_{P}^{-j}$ for $0\leq j\leq m-1$). Then $\delta_P^{-m}$  induces a morphism from $T^{m-1} P^{-m}$ to $P'^\bullet$ whose mapping cone is $P^\bullet$. Therefore, we obtain a distinguished triangle in $\mathcal{K}^b(\A\textup{-proj})$:

\begin{equation}\label{triagproj1}
T^{m-1}P^{-m}\to P'^\bullet\to P^\bullet \to T^mP^{-m}.
\end{equation}

By applying $\Hom_{\mathcal{D}^-(\A)}(V, -)$ to the triangle (\ref{triagproj1}) and using \cite[Prop. I.1.1 (b) \& \S I.6]{hartshorne}, we obtain a long exact sequence

\begin{align*}\label{longexactext1}
\cdots\to \Ext^{i+m-1}(V, P^{-m})\to\Ext^i(V, P'^\bullet)\to\Ext^i(V, P^\bullet)\to\Ext^{i+m}(V, P^{-m})\to \cdots,  
\end{align*} 
where $i>0$ and $\Ext^i=\Ext_{\mathcal{D}^-(\A)}^i$. Note that since $P^{-m}$ is a finitely generated projective $\A$-module and  $P'^\bullet$ is a perfect complex over $\A$ of length smaller than $m$, we obtain by induction that for all $i>0$, $\Ext_{\mathcal{D}^-(\A)}^{i+m}(V, P^{-m})=\Ext_\A^{i+m}(V, P^{-m})=0=\Ext_{\mathcal{D}^-(\A)}^i(V, P'^\bullet)$, which implies $\Ext_{\mathcal{D}^-(\A)}^i(V, P^\bullet)=0$.  

(ii). We can assume without losing generality that $V^\bullet$ is of the form
\begin{align*}
V^\bullet&: 0\to V^{-n}\xrightarrow{\delta_V^{-n}}V^{-n+1}\xrightarrow{\delta_V^{-n+1}}V^{-n+2}\to \cdots\to V^{-1}\xrightarrow{\delta^{-1}_V} V^0\to 0.
\end{align*}
If $n=0$, then $V^\bullet$ is a Gorenstein projective left $\A$-module concentrated in degree zero. In this situation, (ii) follows from (i). Assume then the result true for all complexes $W^\bullet$ whose terms are all in $\A\textup{-GProj}$ and whose length is strictly less than $n$. As in the proof of (i), we obtain a distinguished triangle in $\mathcal{K}^-(\A)$:
\begin{equation}\label{triaggoren2}
T^{n-1}V^{-n}\to V'^\bullet\to V^\bullet \to T^nV^{-n}, 
\end{equation}
where $V'^\bullet$ is the associated truncated complex to $V^\bullet$. As before, by applying $\Hom_{\mathcal{D}^-(\A)}(-, P^\bullet)$ to (\ref{triaggoren2}) and using \cite[Prop. I.1.1 (b),\S I.6]{hartshorne}, for all $i>0$, we obtain a long exact sequence
\begin{align*}
\cdots\to \Ext^{i+n}(V^{-n}, P^\bullet)\to\Ext^i(V^\bullet, P'^\bullet)\to\Ext^i(V'^\bullet, P^\bullet)\to\Ext^{i+n-1}(V^{-n}, P^\bullet)\to \cdots. 
\end{align*} 
Note that since $V^{-n}$ and  $V'^\bullet$ are bounded objects in $\mathcal{C}^-(\A)$ whose terms are in $\A\textup{-GProj}$ and whose length are both smaller than $n$, we obtain by induction that  $\Ext_{\mathcal{D}^-(\A)}^{i+n}(V^{-n}, P^\bullet)=0=\Ext_{\mathcal{D}^-(\A)}^i(V'^\bullet, P^\bullet)$, which implies that $\Ext_{\mathcal{D}^-(\A)}^i(V^\bullet, P^\bullet)=0$ for all $i>0$. This finishes the proof of Lemma \ref{lemma:1.3}.
\end{proof}

Since every projective left $\A$-module is Gorenstein projective, we obtain the following direct consequence of Lemma \ref{lemma:1.3}. 

\begin{corollary}\label{cor:1.6}
If $P^\bullet$ and $Q^\bullet$ are both objects in $\mathcal{K}^b(\A\textup{-proj})$, then $\Ext_{\mathcal{D}^-(\A)}^i(P^\bullet, Q^\bullet)=0$ for all $i>0$. 
\end{corollary}

%By using similar arguments to those in the proof of Lemma \ref{lemma:1.3}, we obtain the proof of Lemma \ref{lemma:1.7} below.

\begin{lemma}\label{lemma:1.7}
Let $R$ be an Artinian object in $\Ca$. Let $M^\bullet$ and $Q^\bullet$ be objects in $\mathcal{K}^-(R\A)$ with $Q^\bullet$ in $\mathcal{K}^b(R\A\textup{-proj})$ and $M^\bullet$ is a {\bf bounded} complex of finitely generated $R\A$-modules. Then $\Ext_{\mathcal{K}^-(R\A)}^i(Q^\bullet, M^\bullet)=0$ for all $i>0$.
\end{lemma}

\subsection{Quasi-lifts, deformations and (uni)versal deformation rings of complexes}

\begin{definition}\label{defi1}
Let $V^\bullet$ be a complex in $\mathcal{D}^-(\A)$ which has only finitely many non-zero cohomology groups, all which have finite $\k$-dimension and let $R\in \Ob(\hat{\Ca})$ be fixed but arbitrary. A {\it quasi-lift} of $V^\bullet$ over $R$ is a pair $(M^\bullet, \phi)$ consisting of a complex $M^\bullet$ in $\mathcal{D}^-(R\A)$ which has finite pseudocompact $R$-tor dimension together with an isomorphism $\phi:\k\hat{\otimes}_R^\mathbf{L}M^\bullet\to V^\bullet$ in $\mathcal{D}^-(\A)$. Two quasi-lifts $(M^\bullet, \phi)$ and $(M'^\bullet, \phi')$ of $V^\bullet$ over $R$ are said to be {\it isomorphic} if there exists an isomorphism $M^\bullet \to M'^\bullet$ in $\mathcal{D}^-(R\A)$ which carries $\phi$ to $\phi'$. A {\it deformation} of $V^\bullet$ over $R$ is an isomorphism class of quasi-lifts of $V^\bullet$ over $R$. We denote by $\mathrm{Def}_\A (V^\bullet, R)$ the set of all deformations of $V^\bullet$ over $R$.
The {\it deformation functor} $\hat{\Fun}_{V^\bullet}:\hat{\Ca}\to \mathrm{Sets}$ associated to $V^\bullet$ is defined as follows. For all $R\in \Ob(\hat{\Ca})$, $\hat{\Fun}_{V^\bullet}(R)=\mathrm{Def}_\A(V^\bullet,R)$, and for all morphisms $\alpha:R\to R'$ in $\hat{\Ca}$, $\hat{\Fun}_{V^\bullet}(\alpha): \mathrm{Def}_\A(V^\bullet,R)\to \mathrm{Def}_\A(V^\bullet,R')$ is the set map induced by $(M^\bullet, \phi)\mapsto (R'\hat{\otimes}_{R,\alpha}^\mathbf{L}M^\bullet, \phi_\alpha)$, where $\phi_\alpha$ denotes the composition $\k\hat{\otimes}^\mathbf{L}_{R'}(R'\hat{\otimes}^\mathbf{L}_{R,\alpha}M^\bullet)\cong \k\hat{\otimes}_R^\mathbf{L}M^\bullet\xrightarrow{\phi}V^\bullet$. We denote by $\Fun_{V^\bullet}$ the restriction of $\hat{\Fun}_{V^\bullet}$ to $\Ca$. It was proved in \cite[Thm. 1.1]{blehervelez2} that $\Fun_{V^\bullet}$ has a pro-representable hull $R(\A, V^\bullet)\in \Ob(\Ca)$ in the sense of \cite[Def. 2.7]{sch} and that $\hat{\Fun}_{V^\bullet}$ is continuous (see \cite{mazur} and \cite[Prop. 2.4.4]{blehervelez2}). Then there exists a quasi-lift $(U(\A,V^\bullet), \phi_U)$ of $V^\bullet$ over $R(\A,V^\bullet)$ with the following property. For each $R\in \Ob(\hat{\Ca})$, the map $\Hom_{\hat{\Ca}}(R(\A,V^\bullet),R)\to \hat{\Fun}_{V^\bullet}(R)$ induced by $\alpha\mapsto (R\hat{\otimes}^\mathbf{L}_{R(\A,V^\bullet), \alpha}U(\A,V^\bullet), \phi_{U,\alpha})$ is surjective, and this map is bijective if $R$ is the ring of dual numbers $\k[\epsilon]$ over $\k$ where $\epsilon^2=0$. The ring $R(\A, V^\bullet)$ is called the {\it versal deformation ring} of $V^\bullet$, which is unique up to a non-canonical isomorphism. If $\hat{\Fun}_{V^\bullet}$ is represented by $R(\A, V^\bullet)$, then we call $R(\A, V^\bullet)$ the {\it universal deformation ring} of $V^\bullet$, which is uniquely determined up to a canonical isomorphism. See \cite[Definitions 2.1.7 \& 2.1.14]{blehervelez2} for more details.
\end{definition}

It follows from \cite[Prop. 2.5.1]{blehervelez2} that if $V^\bullet$ has exactly one non-zero cohomology group $C$ of finite $\k$-dimension, then the versal deformation ring $R(\A,V^\bullet)$ coincides with the versal deformation ring $R(\A, C)$ considered in \cite{blehervelez}. In particular, the groups $\Hom_{\mathcal{D}^-(\A)}(V^\bullet,V^\bullet)$ and $\Hom_\A(C,C)$ are isomorphic. The cases for when $V^\bullet$ is a two-term complex and for when $V^\bullet$ is a completely split complex are also considered in \cite[\S 2.5]{blehervelez2}.  

\begin{remark}\label{rem:1.3}
Assume that $V^\bullet$ is a {\bf bounded} object in $\mathcal{D}^-(\A)$ such that all its cohomology groups have finite $\k$-dimension, let $R$ be an arbitrary Artinian object in $\Ca$, and let $(M^\bullet, \phi)$ be a quasi-lift of $V^\bullet$ over $R$ as in Definition \ref{defi1}.  
%Following \cite[Def. 2.2.1 \& Remark 2.2.2]{blehervelez2}, we define by $\mathcal{C}^-_\textup{fin}(R\A)$ (resp. $\mathcal{K}^-_\textup{fin}(R\A)$, resp. $\mathcal{D}^-_\textup{fin}(R\A)$) to be the full subcategory of  $\mathcal{C}^-(R\A)$ (resp. $\mathcal{K}^-(R\A)$, resp. $\mathcal{D}^-(R\A)$) whose objects are those complexes $X^\bullet$ of finite pseudocompact $R$-tor dimension having finitely many non-zero cohomology groups, all of which have finite $R$-length. 
By \cite[Remarks 2.2.2, 2.2.6, 2.3.5 \& 2.3.6]{blehervelez2}, we have the following.
\begin{enumerate}
\item We can replace $V^\bullet$ for a bounded above complex $\tilde{V}^\bullet$ of abstractly free finitely generated $\A$-modules such that $\tilde{V}^\bullet$ is isomorphic to $V^\bullet$ in $\mathcal{D}^-(\A)$.
\item We can replace $M^\bullet$ for a bounded above complex $\tilde{M}^\bullet$ of abstractly free finitely generated $R\A$-modules such that $\tilde{M}^\bullet$ and $M^\bullet$ are isomorphic in $\mathcal{D}^-(R\A)$, and that the isomorphism $\phi:\k\hat{\otimes}^\mathbf{L}_RM^\bullet\to V^\bullet$ can be replaced by a quasi-isomorphism $\phi': \k\hat{\otimes}_R \tilde{M}^\bullet \to \tilde{V}^\bullet$ in $\mathcal{C}^-(R\A)$ which is surjective on terms, and where $\tilde{V}^\bullet$ is as in (i).
\item Let $X^\bullet$ is an object of $\mathcal{D}^-(R\A)$ such that $X^\bullet$ is of finite $R$-tor dimension having finitely many non-zero cohomology groups, all of which have finite $R$-length. Assume further that all the terms of $X^\bullet$ have finite $\k$-length. If $f:M^\bullet\to X^\bullet$ is an isomorphism in $\mathcal{D}^-(R\A)$, then we can replace $X^\bullet$ for a bounded above complex $\tilde{X}^\bullet$ of abstractly free finitely generated $R\A$-modules such that $\tilde{X}^\bullet$ is isomorphic to $X^\bullet$ in $\mathcal{D}^-(R\A)$, and replace $f$ for a quasi-isomorphism $\tilde{f}:\tilde{M}^\bullet \to \tilde{X}^\bullet$ in $\mathcal{C}^-(R\A)$ that is surjective on terms, where $\tilde{M}^\bullet$ is as in (ii).     
\end{enumerate}
\end{remark}

\section{Representability of the Deformation Functor, (Uni)versal Deformation Rings and Perfect Complexes}\label{section3}
In this section, our goal is to prove Theorem \ref{thm01} by proving Proposition \ref{prop2} and Lemma \ref{lemma:2.10} below. 

\begin{proposition}\label{prop2}
Let $V^\bullet$ be a {\bf bounded} complex in $\mathcal{C}^-(\A)$ whose terms are all in $\A\textup{-GProj}$ and which has only finitely many nonzero cohomology groups, all of which have finite $\k$-dimension. If $\Hom_{\mathcal{D}_\mathfrak{perf}^-(\A)}(V^\bullet, V^\bullet)=\k$, then the versal deformation ring $R(\A,V^\bullet)$ (as in Definition \ref{defi1}) is universal. In particular, the deformation functor $\hat{\Fun}_{V^\bullet}(-)$ is representable. 
\end{proposition}

\begin{remark}\label{rem2}
Let $V^\bullet$ be as in the hypothesis of Proposition \ref{prop2}. Since the functor $\hat{\Fun}_{V^\bullet}$ is continuous by \cite[Prop. 2.2.4]{blehervelez2}, most of the arguments used to prove Proposition \ref{prop2} can be carried out for the restriction $\Fun_{V^\bullet}$ of $\hat{\Fun}_{V^\bullet}$ to the full subcategory $\Ca$ of $\hat{\Ca}$ of Artinian objects by using {\it small extensions} in $\Ca$, i.e., surjections $\pi:R\to R_0$ of Artinian objects in $\Ca$ such that the kernel of $\pi$ is a principal ideal $tR$ annihilated by the maximal ideal $\mathfrak{m}_R$ of $R$. 
\end{remark}

\begin{proof}
We prove Proposition \ref{prop2} by carefully adjusting some of the arguments in the proof of \cite[Thm. 2.6]{blehervelez} to our context.

\begin{claim}\label{claim1}
Let $\pi:R\to R_0$ be a surjection of Artinian objects in $\Ca$. Let $M^\bullet$ be an object in  $\mathcal{C}^-(R\A)$ such that all its terms are abstractly free finitely generated $R\A$-modules, and let $Q^\bullet$ be an arbitrary complex in $\mathcal{K}^b(R\A\textup{-proj})$. Let $M_0^\bullet$ and $Q^\bullet_0$  be objects in $\mathcal{C}^-(R_0\A)$ with $Q^\bullet_0$ a bounded complex such that there are quasi-isomorphisms $g:R_0\hat{\otimes}_{R,\pi} M^\bullet\to M_0^\bullet$ and $h:R_0\hat{\otimes}_{R,\pi} Q^\bullet\to Q_0^\bullet$ in $\mathcal{C}^-(R_0\A)$ that are surjective on terms. Assume further that $\k\hat{\otimes}_R M^\bullet$ is isomorphic to $V^\bullet$ in $\mathcal{D}^-(\A)$.  If $v_0\in \Hom_{\mathcal{K}^-(R_0\A)}(M_0^\bullet,Q_0^\bullet)$, then there exists $v\in \Hom_{\mathcal{K}^-(R\A)}(M^\bullet,Q^\bullet)$ with $v_0\circ g=h\circ(R_0\hat{\otimes}_{R,\pi}v)$.
\end{claim}
\begin{proof}
Since $R$ and $R_0$ are Artinian, we can assume that $\pi:R\to R_0$ is a small extension as in Remark \ref{rem2}. Consider the short exact sequence in $\mathcal{C}^-(R\A)$
\begin{equation}\label{eqn1}
0\to tQ^\bullet\to Q^\bullet\xrightarrow{h\circ \tau_{Q^\bullet}} Q_0^\bullet\to 0,
\end{equation} 
where $\tau_{Q^\bullet}: Q^\bullet \to R_0\hat{\otimes}_{R,\pi}Q^\bullet$ is the natural morphism in $\mathcal{C}^-(R\A)$ which is surjective on terms.  After applying $\Hom_{\mathcal{D}^-(R\A)}(M^\bullet,-)=\Hom_{\mathcal{K}^-(R\A)}(M^\bullet,-)$ to (\ref{eqn1}) and using \cite[Prop. I.6.1]{hartshorne}, we obtain a long exact sequence 
\begin{equation}\label{eqn2}
\cdots\to \Hom_{\mathcal{K}^-(R\A)}(M^\bullet,Q^\bullet)\xrightarrow{(h\circ \tau_{Q^\bullet})_\ast}\Hom_{\mathcal{K}^-(R\A)}(M^\bullet,Q_0^\bullet)\to \Ext_{\mathcal{D}^-(R\A)}^1(M^\bullet,tQ^\bullet)\to \cdots.
\end{equation}
Since the terms of $M^\bullet$ are abstractly free finitely generated $R\A$-modules and $tQ^\bullet \cong \k\hat{\otimes}_RQ^\bullet$ in $\mathcal{C}^-(\A)$, we obtain an isomorphism of abelian groups 
\begin{equation*}
\Ext_{\mathcal{D}^-(R\A)}^1(M^\bullet,tQ^\bullet)=\Hom_{\mathcal{K}^-(R\A)}(M^\bullet, T(tQ^\bullet))\cong \Hom_{\mathcal{K}^-(\A)}(\k\hat{\otimes}_RM^\bullet, T(\k\hat{\otimes}_RQ^\bullet)).
\end{equation*}
Note that in particular that  $P^\bullet=\k\hat{\otimes}_RQ^\bullet$ is a perfect complex over $\A$. Since by hypothesis $\k\hat{\otimes}_RM^\bullet$ and $V^\bullet$ are isomorphic in $\mathcal{D}^-(\A)$, we obtain an isomorphism of abelian groups 
\begin{align*}
\Hom_{\mathcal{K}^-(\A)}(\k\hat{\otimes}_RM^\bullet, T(\k\hat{\otimes}_RQ^\bullet))&=\Hom_{\mathcal{D}^-(\A)}(\k\hat{\otimes}_RM^\bullet, T(\k\hat{\otimes}_RQ^\bullet))\\
&\cong \Hom_{\mathcal{D}^-(\A)}(V^\bullet, TP^\bullet)\\
&=\Ext_{\mathcal{D}^-(\A)}^1(V^\bullet, P^\bullet).
\end{align*}
By using Lemma \ref{lemma:1.3}(ii) we obtain that $\Ext_{\mathcal{D}^-(\A)}^1(V^\bullet,P^\bullet)=0$, which implies 
that the map $(h\circ\tau_{Q^\bullet})_\ast$ in (\ref{eqn2}) is surjective. Let $\tau_{M^\bullet}:M^\bullet\to R_0\hat{\otimes}_{R.\pi}M^{\bullet}$ be the natural morphism in $\mathcal{C}^-(R\A)$ which is surjective on terms. 
Note that since $tR$ is annihilated by $\mathfrak{m}_R$, it follows that all the terms of $Q^\bullet_0$ are also annihilated by $t$. Therefore $\tau_{M^\bullet}$ induces an isomorphism 
\begin{equation*}
\Hom_{\mathcal{K}^-(R_0\A)}(M^\bullet_0,Q^\bullet_0)\xrightarrow{(g\circ\tau_M)^\ast} \Hom_{\mathcal{K}^-(R\A)}(M^\bullet,Q^\bullet_0),
\end{equation*}
and thus there exists a morphism $v:M^\bullet\to Q^\bullet$ in $\mathcal{K}^-(R\A)$ such that  $v_0\circ g=h\circ(R_0\hat{\otimes}_{R,\pi}v)$. 
This finishes the proof of Claim \ref{claim1}.
\end{proof}

\begin{claim}\label{claim2}
Let $M^\bullet$ (resp. $M_0^\bullet$) and $N^\bullet$ (resp. $N_0^\bullet$) be objects in $\mathcal{C}^-(R\A)$ (resp. $\mathcal{C}^-(R_0\A$)) such that there are quasi-isomorphisms $g_M:R_0\hat{\otimes}_{R,\pi}M^\bullet\to M_0^\bullet$ and $g_N:R_0\hat{\otimes}_{R,\pi}N^\bullet\to N_0^\bullet$ in $\mathcal{C}^-(R_0\A)$ that are surjective on terms. Suppose that all the terms of $M^\bullet$ are abstractly free finitely generated $R\A$-modules, that $N^\bullet$ has finite $R$-tor dimension and all its terms are finitely generated $R\A$-modules, and that $\k\hat{\otimes}_R M^\bullet$ and $V^\bullet$ are isomorphic in $\mathcal{D}^-(\A)$. If $\sigma_0\in \Hom_{\mathcal{K}^-(R_0\A)}(M_0^\bullet, N_0^\bullet)$ factors through an object in $\mathcal{K}^b(R_0\A\textup{-proj})$, then there exists a morphism $\sigma\in \Hom_{\mathcal{K}^-(R\A)}(M^\bullet, N^\bullet)$ such that $\sigma$ factors through an object in $\mathcal{K}^b(R\A\textup{-proj})$ and $\sigma_0\circ g_M=g_N\circ (R_0\hat{\otimes}_{R,\pi}\sigma)$. 
\end{claim}

\begin{proof}
Assume that $\sigma_0:M_0^\bullet\to N_0^\bullet$ factors through a complex $Q_0^\bullet$ in $\mathcal{K}^b(R_0\A\textup{-proj})$, namely $\sigma_0=w_0\circ v_0$ with $v_0: M_0^\bullet\to Q_0^\bullet$ and $w_0: Q_0^\bullet \to M_0^\bullet$, where $v_0$ and $w_0$ are suitable morphisms in $\mathcal{K}^-(R_0\A)$. It follows from Lemma \ref{lemma:1.1}(i) that there exists a complex $Q^\bullet$ in $\mathcal{K}^b(R\A\textup{-proj})$ such that $R_0\hat{\otimes}_{R,\pi} Q^\bullet$ is isomorphic to $Q_0$  via some isomorphism $h:R_0\hat{\otimes}_{R,\pi} Q^\bullet\to Q_0^\bullet$ in $\mathcal{C}^-(R_0\A)$.  By Claim \ref{claim1}, there exists $v\in \Hom_{\mathcal{K}^-(R\A)}(M^\bullet, Q^\bullet)$ such that $v_0\circ g_M=h\circ(R_0\hat{\otimes}_{R,\pi}v)$. On the other hand, since $N^\bullet$ has finite $R$-tor dimension, by Remark \ref{remark:1.1} we can assume that $N^\bullet$ is a {\bf bounded} complex whose terms are finitely generated free $R$-modules. Let $\tau_{N^\bullet}: N^\bullet \to R_0\hat{\otimes}_R N^\bullet$  be the natural morphisms in $\mathcal{C}^-(R\A)$ that is surjective on terms. By \cite[Prop. I.6.1]{hartshorne}, the short exact sequence in $\mathcal{C}^-(R\A)$
\begin{equation*}
0\to tN^\bullet\to N^\bullet\xrightarrow{g_N\circ \tau_{N^\bullet}} N_0^\bullet\to 0,
\end{equation*}
induces an exact sequence of abelian groups
\begin{equation*}
\Hom_{\mathcal{K}^-(R\A)}(Q^\bullet,N^\bullet) \xrightarrow{(g\circ\tau_{N^\bullet})_\ast} \Hom_{\mathcal{K}^-(R\A)}(Q^\bullet,N_0^\bullet)\to \Ext_{\mathcal{D}^-(R\A)}^1(Q^\bullet,tN^\bullet).
\end{equation*}
Since by Lemma \ref{lemma:1.7} $\Ext_{\mathcal{D}^-(R\A)}^1(Q^\bullet,tN^\bullet)=0$, it follows that $(g_N\circ \tau_{N^\bullet})_\ast$ is surjective. Then there exists $w\in \Hom_{\mathcal{K}^-(R\A)}(Q^\bullet,N^\bullet)$ such that $g_N\circ \tau_{N^\bullet}\circ w=w_0\circ  h \circ \tau_{Q^\bullet}$, where $\tau_{Q^\bullet}: Q^\bullet \to R_0\hat{\otimes}_R Q^\bullet$ is the natural morphism in $\mathcal{C}^-(R\A)$ which is surjective on terms. In particular, we obtain that $g_N\circ (R_0\hat{\otimes}_{R,\pi}w)=w_0\circ h$.  Let $\sigma = w\circ v\in \Hom_{\mathcal{K}^-(R\A)}(M^\bullet,N^\bullet)$. Then, $\sigma_0\circ g_M=w_0\circ v_0\circ g_M=w_0\circ h \circ (R_0\hat{\otimes}_{R,\pi}v)=g_N\circ (R_0\hat{\otimes}_{R,\pi}w)\circ (R_0\hat{\otimes}_{R,\pi}v)=g_N\circ (R_0\hat{\otimes}_{R,\pi}\sigma)$. This finishes the proof of Claim \ref{claim2}.
\end{proof}

\begin{claim}\label{claim3}
Let $R$ be an Artinian object in $\Ca$, and let $(M^\bullet, \phi)$ and $(M'^\bullet, \phi')$ be two quasi-lifts of $V^\bullet$ over $R$. If there exists an isomorphism $f:M^\bullet \to M'^\bullet$ in $\mathcal{D}^-(R\A)$, then there exists an isomorphism $f': M^\bullet \to M'^\bullet$ in $\mathcal{D}^-(R\A)$ such that $\phi'\circ(\k\hat{\otimes}_R^\mathbf{L}f')=\phi$. In particular, $[M^\bullet, \phi]=[M'^\bullet, \phi']$ in $\Fun_{V^\bullet}(R)$. 
\end{claim}
\begin{proof}
Let $\bar{f}'':V^\bullet \to V^\bullet$ be the isomorphism $\phi\circ(\k\hat{\otimes}_R^\mathbf{L}f)^{-1}\circ(\phi')^{-1}$ in $\mathcal{D}^-(\A)$. Note that the image of $\bar{f}''$ in $\mathcal{D}_\mathfrak{perf}^-(\A)$ can be represented by a diagram in $\mathcal{D}^-(\A)$ of the form 

\begin{equation*}
\begindc{\commdiag}[\Size]
\obj(1,-1)[p0]{$V^\bullet$}
\obj(0,0)[p1]{$V^\bullet$}
\obj(2,0)[p2]{$V^\bullet$}
\mor{p1}{p0}{$\bar{f}''$}[-1,0]
\mor{p2}{p0}{$1_{V^\bullet}$}
\enddc
\end{equation*} 
Since $\Hom_{\mathcal{D}_\mathfrak{perf}^-(\A)}(V^\bullet, V^\bullet)=\k$, it follows that there exists an scalar $\bar{s}_f\in \k$ such that the following diagram in $\mathcal{D}^-(\A)$ is commutative (see e.g. \cite[\S 3]{krause3}): 
%\Hom_{\mathcal{D}_\textup{sg}(\A\textup{-mod})}(V^\bullet, V^\bullet)=
\begin{equation*}
\begindc{\commdiag}[\Size]
\obj(3,3)[p0]{$V^\bullet$}
\obj(2,2)[p1]{$V^\bullet$}
\obj(3,2)[p2]{$V'^\bullet$}
\obj(4,2)[p3]{$V^\bullet$}
\obj(3,1)[p4]{$V^\bullet$}
\mor{p1}{p0}{$\bar{s}_f \cdot 1_{V^\bullet}$}
\mor{p1}{p2}{}
\mor{p0}{p2}{$v$}
\mor{p3}{p2}{$v$}[-1,0]
\mor{p1}{p4}{$\bar{f}''$}[-1,0]
\mor{p3}{p4}{$1_{V^\bullet}$}
\mor{p3}{p0}{$1_{V^\bullet}$}[-1,0]
\mor{p4}{p2}{$v$}
\enddc
\end{equation*}
%where the mapping cone of $v:V^\bullet \to V'^\bullet$ is a perfect complex over $\A$.
where $v:V^\bullet \to V'^\bullet$ fits into an exact triangle $V^\bullet\xrightarrow{v} V'^\bullet \to W^\bullet \to TV^\bullet$ in $\mathcal{D}^-(\A)$, with $W^\bullet$ isomorphic to a perfect complex over $\A$.
By Remark \ref{rem:1.3}, we can assume that $V^\bullet$ is a bounded above complex of abstractly free finitely generated $\A$-modules, that $M^\bullet$ and $M'^\bullet$ are bounded above complexes of abstractly free finitely generated $R\A$-modules, that $f$ is given by a quasi-isomorphism in $\mathcal{C}^-(R\A)$ that is surjective on terms, and that $\phi$, $\phi'$ and $\bar{f}''$ are given by quasi-isomorphisms in $\mathcal{C}^-(\A)$ that are surjective on terms.  Moreover, $v$ can be assumed to be a morphism in $\mathcal{K}^-(\A)$ whose mapping cone is isomorphic to an object in $\mathcal{K}^b(\A\textup{-proj})$. Since $v\circ(\bar{f}''-\bar{s}_f\cdot 1_{V^\bullet})=0$, it follows 
that $\bar{\sigma}_f=\bar{f}''-\bar{s}_f\cdot 1_{V^\bullet}\in\Hom_{\mathcal{K}^-(\A)}(V^\bullet, V^\bullet)$ factors through an object in $\mathcal{K}^b(\A\textup{-proj})$.  By Claim \ref{claim2}, there exists a morphism $\sigma_f\in \Hom_{\mathcal{K}^-(R\A)}(M^\bullet, M^\bullet)$ factoring through an object in $\mathcal{K}^b(R\A\textup{-proj})$ such that $\bar{\sigma}_f\circ \phi=\phi \circ(\k\hat{\otimes}_R \sigma_f)$. Let $s_f\in R$ be such that $\k\hat{\otimes}_Rs_f=\bar{s}_f$, and let $f''=s_f\cdot 1_{M^\bullet} +\sigma_f$. Then $\phi\circ (\k\hat{\otimes}_R f'')=\bar{f}''\circ\phi$ is a quasi-isomorphism in $\mathcal{C}^-(\A)$, and hence so is $\k\hat{\otimes}_R f''$. It follows from Lemma \ref{lemma:1.1}(ii) that $f'':M^\bullet \to M^\bullet$ is a quasi-isomorphism in $\mathcal{C}^-(R\A)$. If we define $f'=f\circ f'':M^\bullet\to M'^\bullet$, then $f'$ is a quasi-isomorphism and $\phi'\circ(\k\hat{\otimes}_R f')=\phi$. This finishes the proof of Claim \ref{claim3}. 
\end{proof}

\begin{claim}\label{claim4}
If $(M^\bullet, \phi)$ is a quasi-lift of $V^\bullet$ over an Artinian object $R$ in $\Ca$, then the ring homomorphism $\eta_{M^\bullet}:R\to \Hom_{\mathcal{D}_{\mathfrak{perf}}^-(R\A)}(M^\bullet, M^\bullet)$ is surjective, where for all $r\in R$, $\eta_{M^\bullet}(r)$ is represented by the diagram 
\begin{equation*}
\begindc{\commdiag}[\Size]
\obj(1,-1)[p0]{$M^\bullet$}
\obj(0,0)[p1]{$M^\bullet$}
\obj(2,0)[p2]{$M^\bullet$}
\mor{p2}{p0}{$1_{M^\bullet}$}[1,0]
\mor{p1}{p0}{$r\cdot 1_{M^\bullet}$}[-1,0]
\enddc
\end{equation*}
\end{claim}

\begin{proof}
Let $\pi:R\to R_0$ be a small extension in $\Ca$ as in Remark \ref{rem2} and assume that the ring homomorphism $\eta_{M_0^\bullet}: R_0\to \Hom_{\mathcal{D}_{\mathfrak{perf}}^-(R_0\A)}(M_0^\bullet, M_0^\bullet)$ is surjective whenever $M_0^\bullet$ defines a quasi-lift of $V^\bullet$ over $R_0$.   Let $(M^\bullet,\phi)$ be a quasi-lift of $V^\bullet$ over $R$ and let $\hat{f}:M^\bullet\to M^\bullet$ be a morphism in $\mathcal{D}_{\mathfrak{perf}}^-(R\A)$. It follows that $\hat{f}$ can be represented by a diagram in $\mathcal{D}^-(R\A)$ of the form
\begin{equation}\label{morphisms}
\begindc{\commdiag}[\Size]
\obj(1,-1)[p0]{$M'^\bullet$}
\obj(0,0)[p1]{$M^\bullet$}
\obj(2,0)[p2]{$M^\bullet$}
\mor{p1}{p0}{$f$}[-1,0]
\mor{p2}{p0}{$s$}
\enddc
\end{equation}
where $f$ is a morphism in $\mathcal{D}^-(R\A)$ and $s:M^\bullet\to M'^\bullet$ fits into an exact triangle $M^\bullet\xrightarrow{s} M'^\bullet \to N^\bullet\to TM^\bullet$ with $N^\bullet$ isomorphic to a perfect complex over $R\A$. As in the proof of Claim \ref{claim3} we use Remark \ref{rem:1.3} so that we can assume that $V^\bullet$ is a bounded above complex of abstractly free finitely generated $\A$-modules, that $M^\bullet$ is a bounded above complex of abstractly free finitely generated $R\A$-modules, that $f$ is a morphism in $\mathcal{K}^-(R\A)$, and that $\phi$ is given by a quasi-isomorphism in $\mathcal{C}^-(R\A)$ that is surjective on terms. We can also assume that $C_s^\bullet$ is isomorphic to an object in $\mathcal{K}^b(R\A\textup{-proj})$.  In particular, the terms of $M'^\bullet$ are finitely generated $R\A$-modules and $C_s^\bullet$ is isomorphic to a bounded complex whose terms are flat as $R$-modules, which implies that $C_s^\bullet$ also has finite $R$-tor dimension (see \cite[Prop. 5.1]{illusie1} and \cite[\S 8.3.6.2]{illusie2}). Since $M^\bullet$ has finite $R$-tor dimension, it follows from \cite[\S 5.3]{illusie1} that $M'^\bullet$ has also finite $R$-tor dimension. On the other hand, note that $R_0\hat{\otimes}_{R,\pi} C_s^\bullet$ is a perfect complex over $R_0\A$ and that $f$ and $s$ induce morphisms $f_0, s_0: M_0^\bullet\to M_0'^\bullet$ in $\mathcal{K}^-(R_0\A)$, where $M_0^\bullet = R_0\hat{\otimes}_{R,\pi}M^\bullet$ and $M_0'^\bullet = R_0\hat{\otimes}_{R,\pi}M'^\bullet$. Let $C_{s_0}^\bullet$ be the mapping cone of $s_0$ in $\mathcal{K}^-(R_0\A)$. By the axiom (TR2) of triangulated categories (see e.g. \cite[\S I.1]{hartshorne}) and by \cite[Prop. I.1.1 (c)]{hartshorne}, there exists an isomorphism $h: R_0\hat{\otimes}_{R,\pi} C_s^\bullet\to C_{s_0}^\bullet$ such that the following diagram of triangles in $\mathcal{K}^-(R_0\A)$ is commutative: 
\begin{equation*}
\begindc{\commdiag}[\Sized]
\obj(0,2)[p0]{$R_0\hat{\otimes}_{R, \pi}M^\bullet$}
\obj(4,2)[p1]{$R_0\hat{\otimes}_{R, \pi}M'^\bullet$}
\obj(8,2)[p2]{$R_0\hat{\otimes}_{R, \pi}C_s^\bullet$}
\obj(12,2)[p3]{$R_0\hat{\otimes}_{R, \pi}TM^\bullet$}
\obj(0,0)[q0]{$M_0^\bullet$}
\obj(4,0)[q1]{$M_0'^\bullet$}
\obj(8,0)[q2]{$C_{s_0}^\bullet$}
\obj(12,0)[q3]{$TM_0^\bullet$}
\mor{p0}{p1}{}
\mor{p1}{p2}{}
\mor{p2}{p3}{}
\mor{q0}{q1}{}
\mor{q1}{q2}{}
\mor{q2}{q3}{}
\mor{p0}{q0}{$=$}
\mor{p1}{q1}{$=$}
\mor{p2}{q2}{$h$}
\mor{p3}{q3}{$=$}
\enddc
\end{equation*}   
It follows that $C_{s_0}^\bullet$ is also a perfect complex over $R_0\A$, and thus that the diagram
\begin{equation*}
\begindc{\commdiag}[\Size]
\obj(1,-1)[p0]{$M_0'^\bullet$}
\obj(0,0)[p1]{$M_0^\bullet$}
\obj(2,0)[p2]{$M_0^\bullet$}
\mor{p2}{p0}{$s_0$}[1,0]
\mor{p1}{p0}{$f_0$}[-1,0]
\enddc
\end{equation*}
represents a morphism $\hat{f}_0: M_0^\bullet \to M_0^\bullet$ in $\mathcal{D}_{\mathfrak{perf}}^-(R_0\A)$. By induction, there exists $r_0\in R_0$ such that $\eta_{M_0^\bullet}(r_0)= \hat{f}_0$ in $\mathcal{D}_{\mathfrak{perf}}^-(R_0\A)$. Therefore, there exists $M_0''^\bullet$, and morphisms  $g_0: M_0^\bullet\to M_0''^\bullet$, $u_0: M_0^\bullet\to M_0''^\bullet$ and $v_0: M_0'^\bullet\to M_0''^\bullet$ in $\mathcal{K}^-(R_0\A)$, such that the mapping cone $C_{s_0}^\bullet$ of $s_0$ is a perfect complex over $R_0\A$ and such that the following diagram in $\mathcal{K}^-(R_0\A)$ is commutative:
\begin{equation*}\label{diagcomm1}
\begindc{\commdiag}[\Size]
\obj(3,3)[p0]{$M_0^\bullet$}
\obj(2,2)[p1]{$M_0^\bullet$}
\obj(3,2)[p2]{$M_0''^\bullet$}
\obj(4,2)[p3]{$M_0^\bullet$}
\obj(3,1)[p4]{$M_0'^\bullet$}
\mor{p1}{p0}{$r_0\cdot 1_{M^\bullet_0}$}
\mor{p1}{p2}{$g_0$}
\mor{p0}{p2}{$u_0$}
\mor{p3}{p2}{$u_0$}[-1,0]
\mor{p1}{p4}{$f_0$}[-1,0]
\mor{p3}{p4}{$s_0$}
\mor{p3}{p0}{$1_{M_0^\bullet}$}[-1,0]
\mor{p4}{p2}{$v_0$}
\enddc
\end{equation*}
By Lemma \ref{lemma:1.1}(iii), the mapping cone $C_{v_0}^\bullet$ of  $v_0$ is also a perfect complex over $R_0\A$. Since  $v_0\circ (s_0\circ(r_0\cdot 1_{M^\bullet_0})-f_0)=0$, it follows 
that the morphism $\sigma_0= s_0\circ (r_0\cdot 1_{M^\bullet_0})-f_0\in \Hom_{\mathcal{K}^-(R_0\A)}(M_0^\bullet, M_0'^\bullet)$ factors through a perfect complex over $R_0\A$. By Claim \ref{claim3}, there exists a morphism $\sigma\in  \Hom_{\mathcal{K}^-(R\A)}(M^\bullet, M'^\bullet)$ which factors through a perfect complex over $R\A$ such that $\sigma_0=R_0\hat{\otimes}_{R,\pi}\sigma$.  Let $r\in R$ be such that $\pi(r)=r_0$. It follows that
$r_0\cdot 1_{M^\bullet_0}=R_0\hat{\otimes}_{R,\pi}(r\cdot 1_{M^\bullet})$.  Consider the morphism $h=s\circ (r\cdot 1_{M^\bullet})-f-\sigma: M^\bullet \to M'^\bullet$ in $\mathcal{K}^-(R\A)$. It follows that $h_0=R_0\hat{\otimes}_{R,\pi}h$ is equal to the zero morphism in $\mathcal{K}^-(R_0\A)$. Since as noted above, the terms of $M'^\bullet$ are all abstractly free over $R\A$, it follows by \cite[Lemma 2.3.3]{blehervelez} (see also \cite[Sublemma VI.8.20]{milne}) that $h$ is equal to a morphism $h_1\in \Hom_{\mathcal{K}^-(R\A)}(M^\bullet, tM'^\bullet)$, where $tR$ is the kernel of $\pi$ as in Remark \ref{rem2}.  
Thus 
\begin{equation}\label{fequals}
s\circ (r\cdot 1_{M^\bullet})-f=\sigma+ h_1.
\end{equation}
Let $\tau_1: tM^\bullet\to V^\bullet$ and $\tau_2: tM'^\bullet \to \k\hat{\otimes}_R M'^\bullet$ be isomorphisms in $\mathcal{C}^-(R\A)$ induced by the isomorphism $tR\cong \k$ in $\Ca$, and let $\tau_{M^\bullet}:M^\bullet \to \k\hat{\otimes}_R M^\bullet$ and $\tau_{M'^\bullet}:M'^\bullet \to \k\hat{\otimes}_R M'^\bullet$ be the natural morphisms in $\mathcal{C}^-(R\A)$ that are surjective on terms. 
%Then  $t\cdot s: tM^\bullet\to tM'^\bullet$ induces a morphism $\k\hat{\otimes}_Rs: V^\bullet\to\k\hat{\otimes}_RM'^\bullet$ in $\mathcal{K}^-(\A)$ such that $(\k\hat{\otimes}_Rs)\circ \tau_1=\tau_2\circ(t\cdot s)$, and whose mapping cone $C_{\k\hat{\otimes}_Rs}^\bullet$ is a perfect complex over $\A$. 
We obtain an isomorphism of abelian groups 
\begin{equation}\label{homotopyiso}
\Hom_{\mathcal{K}^-(R\A)}(M^\bullet, tM'^\bullet)\to \Hom_{\mathcal{K}^-(\A)}(\k\hat{\otimes}_RM^\bullet, \k\hat{\otimes}_R M'^\bullet).
\end{equation}
which sends the homotopy class of a morphism $g:M^\bullet\to tM'^\bullet$  to the homotopy class of $g_1$ when $g_1\circ \tau_{M^\bullet}=\tau_2\circ g$.
Assume that $h_1$ is sent to $\beta_1$ by the isomorphism (\ref{homotopyiso}). Let $s: M^\bullet \to M'^\bullet$ be as in (\ref{morphisms}). Note that $s$ induces a morphism $tM^\bullet \to tM'^\bullet$ which we denote by $s_t$ such that $(\k\hat{\otimes}_Rs)\circ \tau_1=\tau_2\circ s_t$, and since $C_s^\bullet$ is a perfect complex over $R\A$, it follows that $C_{\k\hat{\otimes}_Rs}^\bullet$ is also a perfect complex over $\A$. It follows that both $\beta_1$ and $\k\hat{\otimes}_Rs$ define a morphism in $\mathcal{D}_{\mathfrak{perf}}^-(\A)$
\begin{equation*}
\begindc{\commdiag}[\Size]
\obj(1,-1)[p0]{$\k\hat{\otimes}_RM'^\bullet$}
\obj(0,0)[p1]{$\k\hat{\otimes}_RM^\bullet$}
\obj(2,0)[p2]{$\k\hat{\otimes}_RM^\bullet$}
\mor{p2}{p0}{$\k\hat{\otimes}_Rs$}[1,0]
\mor{p1}{p0}{$\beta_1$}[-1,0]
\enddc
\end{equation*}
Since $\Hom_{\mathcal{D}_{\mathfrak{perf}}^-(\A)}(\k\hat{\otimes}_RM^\bullet, \k\hat{\otimes}M^\bullet)=\Hom_{\mathcal{D}_{\mathfrak{perf}}^-(\A)}(V^\bullet, V^\bullet)=\k$, there exists $\bar{a}\in \k$ such that the following diagram is commutative in 
$\mathcal{K}^-(\A)$:
\begin{equation*}\label{diagcomm2}
\begindc{\commdiag}[\Size]
\obj(3,3)[p0]{$\k\hat{\otimes}_RM^\bullet$}
\obj(2,2)[p1]{$\k\hat{\otimes}_RM^\bullet$}
\obj(3,2)[p2]{$V''^\bullet$}
\obj(4,2)[p3]{$\k\hat{\otimes}_RM^\bullet$}
\obj(3,1)[p4]{$\k\hat{\otimes}_RM'^\bullet$}
\mor{p1}{p0}{$\bar{a}\cdot 1_{\k\hat{\otimes}_RM^\bullet}$}
\mor{p1}{p2}{}
\mor{p0}{p2}{$y_0$}
\mor{p3}{p2}{$y_0$}[-1,0]
\mor{p1}{p4}{$\beta_1$}[-1,0]
\mor{p3}{p4}{$\k\hat{\otimes}_Rs$}
\mor{p3}{p0}{$1_{\k\hat{\otimes}_RM^\bullet}$}[-1,0]
\mor{p4}{p2}{$x_0$}
\enddc
\end{equation*}
where the mapping cone of $y_0$ is a perfect complex over $\A$. By Lemma \ref{lemma:1.3}(iii), the mapping cone of $x_0$ is also a perfect complex over $\A$. Since $x_0\circ ((\k\hat{\otimes}_Rs)\circ (\bar{a}\cdot 1_{\k\hat{\otimes}_RM^\bullet})-\beta_1)=0$, it follows that the morphism $(\k\hat{\otimes}_Rs)\circ \bar{a}\cdot 1_{\k\hat{\otimes}_RM^\bullet}-\beta_1\in \Hom_{\mathcal{K}^-(\A)}(\k\hat{\otimes}_RM^\bullet, \k\hat{\otimes}_RM'^\bullet)$ factors through a perfect complex over $\A$. By Claim \ref{claim2}, there exists $\sigma_1\in \Hom_{\mathcal{K}^-(R\A)}(M^\bullet, M'^\bullet)$ that factors through a perfect complex over $R\A$ such that $(\k\hat{\otimes}_Rs)\circ \bar{a}\cdot 1_{\k\hat{\otimes}_RM^\bullet}-\beta_1=\k\hat{\otimes}_R\sigma_1$ and $(\k\hat{\otimes}_R\sigma_1)\circ \tau_{M^\bullet}=\tau_{M'^\bullet}\circ \sigma_1$. Let $\sigma_2: M^\bullet \to tM'^\bullet$ be the pre-image of $\k\hat{\otimes}_R\sigma_1$ under the isomorphism (\ref{homotopyiso}). Let $a\in R$ such that $ta$ is sent to $\bar{a}$ by the isomorphism $tR\cong \k$. Note that $(\bar{a}\cdot 1_{\k\hat{\otimes}_RM^\bullet})\circ \tau_{M^\bullet}=\tau_2\circ (ta\cdot 1_{M^\bullet})$. Therefore, $\tau_2\circ(s_t\circ (ta\cdot 1_{M^\bullet})-h_1)=\tau_2\circ s_t\circ (ta\cdot 1_{M^\bullet})-\tau_2\circ h_1=(\k\hat{\otimes}_Rs)\circ \tau_1\circ (ta\cdot 1_{M^\bullet})-\beta_1\circ\tau_{M^\bullet}=(\k\hat{\otimes}_Rs)\circ (\bar{a}\cdot 1_{\k\hat{\otimes}_RM^\bullet})\circ\tau_{M^\bullet}-\beta_1\circ\tau_{M^\bullet}=(\k\hat{\otimes}_R\sigma_1)\circ \tau_{M^\bullet}=\tau_2\circ \sigma_2$, and consequently $s_t\circ(ta\cdot 1_{M^\bullet})-h_1=\sigma_2$. On the other hand,  since $\sigma_1$ factors through a perfect complex over $R\A$, there exists a suitable non-trivial morphism $u$ in $\mathcal{K}^-(R\A)$ whose mapping cone is isomorphic to a perfect complex over $R\A$ such that $\sigma_1\circ u=0$. This implies that $0=\tau_{M'^\bullet}\circ \sigma_1\circ u=(\k\hat{\otimes}_R \sigma_1)\circ \tau_{M^\bullet}\circ u= \tau_2\circ \sigma_2\circ u$. Since $\tau_2$ is an isomorphism, it follows that $\sigma_2\circ u=0$, which implies that $\sigma_2$ factors through a perfect complex over $R\A$. By using (\ref{fequals}) and that $s_t$ is the restriction of $s:M^\bullet \to M'^\bullet$ to $tM^\bullet$, we obtain that $s\circ (r\cdot 1_{M^\bullet})-f= \sigma+ s\circ (ta\cdot 1_{M^\bullet})-\sigma_2$, which is equivalent to 
\begin{equation*}\label{fequals2}
s\circ ((r-ta)\cdot 1_{M^\bullet})-f=\sigma-\sigma_2. 
\end{equation*}
Since both  $\sigma$ and $\sigma_2$ factor through a perfect complex over $R\A$, it follows 
that there exists a morphism $v: M'^\bullet \to M''^\bullet$ in $\mathcal{K}^-(R\A)$ whose mapping cone is a perfect complex over $R\A$ such that $v\circ (s\circ ((r-ta)\cdot 1_{M^\bullet})-f)=0$, and thus we obtain a commutative diagram in $\mathcal{K}^-(R\A)$:
\begin{equation*}\label{diagcomm3}
\begindc{\commdiag}[\Size]
\obj(3,3)[p0]{$M^\bullet$}
\obj(2,2)[p1]{$M^\bullet$}
\obj(3,2)[p2]{$M''^\bullet$}
\obj(4,2)[p3]{$M^\bullet$}
\obj(3,1)[p4]{$M'^\bullet$}
\mor{p1}{p0}{$(r-ta)\cdot 1_{M^\bullet}$}
\mor{p1}{p2}{$v\circ f$}
\mor{p0}{p2}{$v\circ s$}
\mor{p3}{p2}{$v\circ s$}[-1,0]
\mor{p1}{p4}{$f$}[-1,0]
\mor{p3}{p4}{$s$}
\mor{p3}{p0}{$1_{M^\bullet}$}[-1,0]
\mor{p4}{p2}{$v$}
\enddc
\end{equation*}
Since the mapping cone of $v\circ s$ is also a perfect complex over $R\A$, it follows from e.g. \cite[\S3.1]{krause3} that $\eta_{M^\bullet}(r-ta)= \hat{f}$ in $\Hom_{\mathcal{D}^-_{\mathfrak{perf}}(R\A)}(M^\bullet, M^\bullet)$, which proves that $\eta_{M^\bullet}$ is surjective. This finishes the proof of Claim \ref{claim4}.
\end{proof}

The following claim, verifies Schlessinger's criterion (H${_4}$) in \cite[Thm. 2.11]{sch} for $\Fun_{V^\bullet}$.

\begin{claim}\label{claim5}
Let $\pi:R'\to R$ be an small extension of Artinian objects in $\Ca$ as in Remark \ref{rem2}. The natural map $\Fun_{V^\bullet}(R'\times_RR')\to \Fun_{V^\bullet}(R')\times_{F_{V^\bullet}(R)}\Fun_{V^\bullet}(R')$ is injective.
\end{claim}
\begin{proof}
Let $\alpha_1:R'\times_RR'\to R'$ (resp. $\alpha_2:R'\times_RR'\to R'$) be the natural surjection onto the first (resp. second) component of $R'\times_RR'$ such that $\pi\circ \alpha_1=\pi\circ \alpha_2$. Let $(M_1^\bullet, \phi_1)$ and $(M_2^\bullet, \phi_2)$ be quasi-lifts of $V^\bullet$ over $R'\times_R R'$. 
Suppose there are isomorphisms $f_i: R'\hat{\otimes}_{R'\times_RR', \alpha_i}^\mathbf{L}M_1^\bullet\to R'\hat{\otimes}_{R'\times_RR', \alpha_i}^\mathbf{L}M_2^\bullet$ in $\mathcal{D}^-(R'\A)$ for $i=1,2$. Then $g_R=(R\hat{\otimes}_{R',\pi}^\mathbf{L}f_2)^{-1}\circ (R\hat{\otimes}_{R,\pi}^\mathbf{L}f_1)$ is an automorphism of $M_R^\bullet=R\hat{\otimes}_{R'\times_RR', \pi\circ \alpha_2}^\mathbf{L} M_1^\bullet$ in $\mathcal{D}^-(R\A)$.  Then $g_R$ defines a morphism in $\mathcal{D}^-_{\mathfrak{perf}}(R\A)$ that can be described by the diagram
\begin{equation*}
\begindc{\commdiag}[\Size]
\obj(1,-1)[p0]{$M_R^\bullet$}
\obj(0,0)[p1]{$M_R^\bullet$}
\obj(2,0)[p2]{$M_R^\bullet$}
\mor{p2}{p0}{$1_{M_R^\bullet}$}[1,0]
\mor{p1}{p0}{$g_R$}[-1,0]
\enddc
\end{equation*}
Since $M_R^\bullet$ defines a quasi-lift of $V^\bullet$ over $R$, it follows from Claim \ref{claim4} that there exists $r\in R$ such that the following diagram in $\mathcal{D}^-(R\A)$ is commutative:
\begin{equation*}\label{diagcomm4}
\begindc{\commdiag}[\Size]
\obj(3,3)[p0]{$M_R^\bullet$}
\obj(2,2)[p1]{$M_R^\bullet$}
\obj(3,2)[p2]{$M_R''^\bullet$}
\obj(4,2)[p3]{$M_R^\bullet$}
\obj(3,1)[p4]{$M_R^\bullet$}
\mor{p1}{p0}{$r\cdot 1_{M^\bullet_R}$}
\mor{p1}{p2}{}
\mor{p0}{p2}{}
\mor{p3}{p2}{$v_R$}[-1,0]
\mor{p1}{p4}{$g_R$}[-1,0]
\mor{p3}{p4}{$1_{M_R^\bullet}$}
\mor{p3}{p0}{$1_{M_R^\bullet}$}[-1,0]
\mor{p4}{p2}{}
\enddc
\end{equation*}
where the mapping cone of $v_R$ is a perfect complex over $R$. Since $v_R\circ (r\cdot 1_{M^\bullet_R}-g_R)=0$, it follows that $r\cdot 1_{M^\bullet_R}-g_R$ factors through a perfect complex over $R\A$. Let $M_{R'}^\bullet = R'\hat{\otimes}_{R'\times_RR', \alpha_2}^\mathbf{L}M_1^\bullet$. Note that $M^\bullet_R=R\hat{\otimes}_{R',\pi}^\mathbf{L}M^\bullet_{R'}$. By Remark \ref{rem:1.3} we can assume that $V^\bullet$ is a bounded above complex of abstractly free finitely generated $\A$-modules, that $M^\bullet_{R'}$ is a bounded above complex of abstractly free finitely generated $R'\A$-modules, and that $g_R$ is given by a quasi-isomorphism in $\mathcal{C}^-(R\A)$ that is surjective on terms.
By Claim \ref{claim2}, there exists a morphism $\sigma_{R'}\in \Hom_{\mathcal{K}^-(R'\A)}(M_{R'}^\bullet, M_{R'}^\bullet)$ that factors through a perfect complex over $R'\A$ such that $r\cdot 1_{M_R^\bullet}-g_R=R\hat{\otimes}_{R',\pi}\sigma_{R'}$. Let $r'\in R'$ such that $\pi(r')=r$, and let $g_{R'} = r'\cdot 1_{M_{R'}^\bullet} - \sigma_{R'}$. Note that $R\hat{\otimes}_{R', \pi}(r'\cdot 1_{M_{R'}^\bullet})=r\cdot 1_{M_R^\bullet}$.  Therefore  
\begin{align*}
\k\hat{\otimes}_{R'}g_{R'}&=\k\hat{\otimes}_R(R\hat{\otimes}_{R', \pi} g_{R'})= \k\hat{\otimes}_R(R\hat{\otimes}_{R', \pi}(r'\cdot 1_{M_{R'}^\bullet}-\sigma_{R'}))=\k\hat{\otimes}_R(r\cdot 1_{M_{R}^\bullet}-r\cdot 1_{M_{R}^\bullet}+g_R)=\k\hat{\otimes}_Rg_R,
\end{align*}
which implies that $\k\hat{\otimes}_{R'}g_{R'}$ is a quasi-isomorphism in $\mathcal{C}^-(\A)$. By Lemma \ref{lemma:1.1}(ii), it follows that $g_{R'}$ defines an automorphism of $M_{R'}^\bullet$ in $\mathcal{D}^-(R'\A)$. By replacing $f_2$ with $f_2\circ g_{R'}$, we have $R\hat{\otimes}_{R', \pi}^\mathbf{L}f_1=R\hat{\otimes}_{R', \pi}^\mathbf{L}f_2$. Therefore the pair $(f_1,f_2)$ defines an isomorphism $f: M_1^\bullet \to M_2^\bullet$ in $\mathcal{D}^-((R'\times_RR')\A)$. By Claim \ref{claim3}, we have that $[M_1^\bullet, \phi_1]=[M_2^\bullet, \phi_2]$ in $\Fun_{V^\bullet}(R'\times_RR')$, which proves Claim \ref{claim5}.
\end{proof}
Since the deformation functor $\hat{\Fun}_{V^\bullet}$ is continuous by \cite[Prop. 2.4.4]{blehervelez2}, Proposition \ref{prop2} follows from \cite[\S 2]{sch} together with \cite[Prop. 2.4.1 \& 2.4.3]{blehervelez2}  and Claim \ref{claim5}. This finishes the proof of Proposition \ref{prop2}.
\end{proof}

\begin{remark}\label{rem:2.7}
Let $P^\bullet$ be an object in $\mathcal{K}^b(\A\textup{-proj})$.  Let $R$ be an Artinian object in $\Ca$, let $\iota_R:\k\to R$ be the unique morphism in $\Ca$ that endows $R$ with a $\k$-algebra structure, and let $\pi_R:R\to \k$ be the natural projection in $\Ca$. Then $\pi_R\circ \iota_R=1_\k$ and $P^\bullet_R=R\hat{\otimes}_\k P^\bullet$ is an object in $\mathcal{K}^b(R\A\textup{-proj})$. In particular, $(P^\bullet_R, \pi_{R, P^\bullet})$ is a quasi-lift of $P^\bullet$ over $R$ , where $\pi_{R, P^\bullet}: \k\hat{\otimes}_{R, \pi_R}(R\hat{\otimes}_{\k, \iota_R} P^\bullet)\to P^\bullet$ is the natural isomorphism in $\mathcal{K}^-(R\A)$.   
\end{remark}
We obtain the following result, which is a version of Claim 6 within the proof of \cite[Thm. 2.6]{blehervelez} (resp. Claim 2.3.6 within the proof of \cite[Thm. 2.2]{velez2}) for derived categories. 

\begin{lemma}\label{lemma:2.8}
Let $R$ be an Artinian object in $\Ca$ and let $P^\bullet_R$ be a perfect complex over $R\A$. Suppose there is a commutative diagram of triangles in $\mathcal{K}^-(R\A)$
\begin{equation}\label{triagproyec}
\begindc{\commdiag}[\Sized]
\obj(0,2)[p0]{$P^\bullet_R$}
\obj(4,2)[p1]{$A^\bullet$}
\obj(8,2)[p2]{$C^\bullet$}
\obj(12,2)[p3]{$TP^\bullet_R$}
\obj(0,0)[q0]{$P^\bullet$}
\obj(4,0)[q1]{$\k\hat{\otimes}_RA^\bullet$}
\obj(8,0)[q2]{$\k\hat{\otimes}_RC^\bullet$}
\obj(12,0)[q3]{$TP^\bullet$}
\mor{p0}{p1}{$g$}
\mor{p1}{p2}{$h$}
\mor{p2}{p3}{}
\mor{q0}{q1}{$\bar{g}$}
\mor{q1}{q2}{$\bar{h}$}
\mor{q2}{q3}{}
\mor{p0}{q0}{}
\mor{p1}{q1}{}
\mor{p2}{q2}{}
\mor{p3}{q3}{}
\enddc
\end{equation}   
where the terms of $A^\bullet$ and $C^\bullet$ are abstractly free finitely generated $R\A$-modules, the complexes $\k\hat{\otimes}_RA^\bullet$ and $\k\hat{\otimes}_RC^\bullet$ are isomorphic in $\mathcal{D}^-(\A)$ to bounded complexes with terms in $\A\textup{-GProj}$, and that the bottom row arises by tensoring the top row with $\k$ over $R$ and by using the isomorphism $\pi_{R, P^\bullet}:\k\hat{\otimes}_{R, \pi_R}P^\bullet_R\to P^\bullet$ in $\mathcal{K}^-(R\A)$. Then the morphism $g$ in the top row of (\ref{triagproyec}) is a section in $\mathcal{K}^-(R\A)$. 
\end{lemma}

\begin{proof}
Without losing generality, assume that %both $\k\hat{\otimes}_RA^\bullet$ and $\k\hat{\otimes}_RC^\bullet$ are bounded complexes with terms in $\mathrm{CM}(\A)$ and that 
$P^\bullet_R$ is an object in $\mathcal{K}^b(R\A\textup{-proj})$.  Note that $P^\bullet = \k\hat{\otimes}_R P^\bullet_R$ is then an object in $\mathcal{K}^b(\A\textup{-proj})$. Assume that $\k\hat{\otimes}_RA^\bullet$ and $\k\hat{\otimes}_RC^\bullet$ are respectively isomorphic to $W$ and $W'$ in $\mathcal{D}^-(\A)$, where $W$ and $W'$ are {\bf bounded} complexes in $\mathcal{C}^-(\A)$ whose terms are in $\A\textup{-GProj}$.  It follows from Lemma \ref{lemma:1.3}(ii) that 
\begin{align*}
%\Hom_{\mathcal{K}^-(\A)}(\k\hat{\otimes}_RA^\bullet, TP^\bullet)&=\Hom_{\mathcal{D}^-(\A)}(\k\hat{\otimes}_RA^\bullet, TP^\bullet)=\Ext_{\mathcal{D}^-(\A)}^1(W, P^\bullet)=0, \text{and}\\
\Hom_{\mathcal{K}^-(\A)}(\k\hat{\otimes}_RC^\bullet, TP^\bullet)&=\Hom_{\mathcal{D}^-(\A)}(\k\hat{\otimes}_RC^\bullet, TP^\bullet)=\Ext_{\mathcal{D}^-(\A)}^1(W, P^\bullet)=0.
\end{align*}
In particular, it follows from \cite[Prop. I.6.1]{hartshorne} that there exists a morphism $\bar{w}: \k\hat{\otimes}A^\bullet \to P^\bullet$ in $\mathcal{K}^-(\A)$ such that $\bar{w}\circ \bar{g}=1_{P^\bullet}$. %By applying $\Hom_{\mathcal{D}^-(\A)}(-,P^\bullet)$ to the bottom row of (\ref{triagproyec}) and by using that all the terms of $\k\hat{\otimes}_RC^\bullet$ are in $\mathrm{CM}(\A)$ together with Lemma \ref{lemma:1.3}(ii) and \cite[Prop. I.6.1]{hartshorne}, we obtain that there exists a morphism $\bar{w}: \k\hat{\otimes}A^\bullet \to P^\bullet$ in $\mathcal{K}^-(\A)$ such that $\bar{w}\circ \bar{g}=1_{P^\bullet}$. 
By Claim \ref{claim1}, there exists a morphism $w: A^\bullet \to P^\bullet_R$ in $\mathcal{K}^-(R\A)$ such that $\k\hat{\otimes}_Rw=\bar{w}$. Hence $\k\hat{\otimes}_R(w\circ g)=\bar{w}\circ \bar{g} = 1_{P^\bullet}$. By Lemma \ref{lemma:1.1}, we obtain that $w\circ g$ is a quasi-isomorphism, which implies that $g$ induces a section in $\mathcal{D}^-(R\A)$. Since the terms of $P_R^\bullet$ and $A^\bullet$ are finitely generated projective $R\A$-modules, we can assume that $g$ is a section in $\mathcal{K}^-(R\A)$. This finishes the proof of Lemma \ref{lemma:2.8}.   
\end{proof}

%\section{Versal Deformation Rings and Perfect Complexes}

We next use Lemma \ref{lemma:2.8} to prove the following result, which is a version of \cite[Lemma 3.2.2]{blehervelez2} and \cite[Lemma 2.2]{bekkert-giraldo-velez} for derived categories. Although its proof can be obtained from that of \cite[Lemma 3.2.2]{blehervelez2} by easily adjusting the arguments to the language of derived categories, we decided to include it for the convenience of the reader.

\begin{lemma}\label{lemma:2.10}
Suppose that $V^\bullet$ is a {\bf bounded} complex in $\mathcal{C}^-(\A)$ such that all its terms are in $\A\textup{-GProj}$, and let $P^\bullet$ be a perfect complex over $\A$. Then $P^\bullet$ has a universal deformation ring $R(\A,P^\bullet)$ isomorphic to $\k$, and the versal deformation ring $R(\A,V^\bullet\oplus P^\bullet)$ is isomorphic to the versal deformation ring $R(\A, V^\bullet)$. Moreover, $R(\A, V^\bullet\oplus P^\bullet)$ is universal if and only if $R(\A,V^\bullet)$ is universal.  
\end{lemma}
\begin{proof}
By Remark \ref{rem:1.3}, we can assume $V^\bullet$ to be a bounded above complex of abstractly free finitely generated $\A$-modules, and that $P^\bullet$ is an object in $\mathcal{K}^b(\A\textup{-proj})$. It follows from Corollary \ref{cor:1.6} that $\Ext_{\mathcal{D}^-(\A)}^1(P^\bullet, P^\bullet)=0$, which implies by \cite[Thm. 2.1.12]{blehervelez2} that the versal deformation ring $R(\A,P^\bullet)$ is isomorphic to $\k$. For each object $R$ in $\hat{\Ca}$, let $\iota_R:\k\to R$ be the unique morphism in $\hat{\Ca}$ that endows $R$ with a $\k$-algebra structure, and let $\pi_R: R\to \k$ be the morphism of $R$ to its residue field $\k$ in $\hat{\Ca}$. In particular, $\pi_R\circ \iota_R=1_\k$. This means that $\k$ is the universal deformation ring of $P^\bullet$. 

Let $R$ be a Noetherian object in $\hat{\Ca}$. Define a map 
\begin{align}\label{mapdef}
\Def_\A(V^\bullet, R)&\to \Def_\A(V^\bullet\oplus P^\bullet, R)\\
[M^\bullet, \phi]&\mapsto [M^\bullet\oplus P^\bullet_R, \phi\oplus \pi_{R,P^\bullet}]\notag
\end{align}
where $P^\bullet_R=R\hat{\otimes}_{\k,\iota_R}P^\bullet$ and $\pi_{R,P^\bullet}: \k\hat{\otimes}_{R, \pi_R} P^\bullet_R\to P^\bullet$ are as in Remark \ref{rem:2.7}. Since for all $\alpha: R\to R'$ in $\Ca$, $\alpha\circ \iota_{R'}$ and $\pi_{R'}\circ \alpha=\pi_R$, it follows that the map (\ref{mapdef}) is natural with respect to morphisms in $\hat{\Ca}$. Due to the continuity of $\hat{\Fun}_{V^\bullet}$ and $\hat{\Fun}_{V^\bullet\oplus P^\bullet}$, in order to prove Lemma \ref{lemma:2.10}, it suffices to show that the map (\ref{mapdef}) is bijective for all Artinian objects $R$ in $\Ca$.  

Let $R$ be an Artinian object in $\Ca$ be fixed but arbitrary. Suppose that $(M^\bullet, \phi)$ and $(M'^\bullet, \phi')$ are two quasi-lifts of $V$  over $R$ such that there exists an isomorphism in $\mathcal{D}^-(R\A)$ 
\begin{equation*}
f=\begin{pmatrix} f_{11} & f_{12}\\f_{21} & f_{22}\end{pmatrix}: M^\bullet\oplus P^\bullet_R\to M'^\bullet\oplus P^\bullet_R
\end{equation*}
with $(\phi'\oplus \pi_{R, P^\bullet})\circ(\k\hat{\otimes}_R^\mathbf{L}f)=\phi\oplus\pi_{R, P^\bullet}$. In particular, $\phi'\circ(\k\hat{\otimes}_R^\mathbf{L}f_{11})=\phi$ and $\k\hat{\otimes}_Rf_{22}=1_{P^\bullet}$. By Remark \ref{lemma:1.3}, we can assume that $V^\bullet$ is a bounded above complex of abstractly free finitely generated $\A$-modules, that both $M^\bullet$ and $M'^\bullet$ are bounded above complexes of abstractly free finitely generated $R\A$-modules, that $f$ is given by a quasi-isomorphism in $\mathcal{C}^-(R\A)$ that is surjective on terms, and that both $\phi$ and $\phi'$ are given by quasi-isomorphisms in $\mathcal{C}^-(\A)$ that are also surjective on terms. By Lemma \ref{lemma:1.1} (ii), it follows that both $f_{11}$ and $f_{22}$ are both quasi-isomorphisms in $\mathcal{C}^-(R\A)$, which induce isomorphisms in $\mathcal{D}^-(R\A)$. Therefore, $[M^\bullet, \phi]= [M'^\bullet, \phi']$ and the map (\ref{mapdef}) is injective.  

We next show that the map (\ref{mapdef}) is surjective. Let $(A^\bullet, \varphi)$ be a lift of $V^\bullet\oplus P^\bullet$ over $R$. As before, we can assume that $A^\bullet$ is a bounded above complex of abstractly free finitely generated $R\A$-modules and that $\varphi$ is given by a quasi-isomorphism in $\mathcal{C}^-(\A)$, which is surjective on terms. Consider $\tau_{A^\bullet}: A^\bullet \to \k\hat{\otimes}_R A^\bullet$ be the natural morphism in $\mathcal{C}^-(R\A)$ that is surjective on terms. Note that $\k\hat{\otimes}_RA^\bullet$ is isomorphic in $\mathcal{D}^-(\A)$ to a {\bf bounded} complex in $\mathcal{C}^-(\A)$ whose terms are all in $\A\textup{-GProj}$. Thus consider the short exact sequence of complexes in $\mathcal{C}^-(R\A)$
\begin{equation*}\label{seqA}
0\to \mathfrak{m}_RA^\bullet\to A^\bullet \xrightarrow{\tau_{A^\bullet}} \k\hat{\otimes}_R A^\bullet\to 0,
\end{equation*}
where $\mathfrak{m}_R$ denotes the unique maximal ideal of $R$. Since $A^\bullet$ has finite $R$-tor dimension, following Remark \ref{remark:1.1}, we can assume that $A^\bullet$ is a {\bf bounded} complex of finitely generated $R\A$-modules which are free over $R$. Hence, we can assume that $\mathfrak{m}_RA^\bullet$ is a {\bf bounded} complex in $\mathcal{C}^-(R\A)$. By \cite[Prop. I.6.1]{hartshorne}, we obtain a long exact sequence of abelian groups
\begin{equation*}
\Hom_{\mathcal{K}^-(R\A)}(P^\bullet_R, \mathfrak{m}_RA^\bullet)\to \Hom_{\mathcal{K}^-(R\A)}(P^\bullet_R, A^\bullet)\xrightarrow{(\tau_{A^\bullet})_\ast} \Hom_{\mathcal{K}^-(R\A)}(P^\bullet_R, \k\hat{\otimes}_R A^\bullet)\to \Ext^1_{\mathcal{D}^-(R\A)}(P^\bullet_R, \mathfrak{m}_RA^\bullet),
\end{equation*}
%where $\Hom=\Hom_{\mathcal{D}^-(R\A)}$ and $\Ext^1=\Ext^1_{\mathcal{D}^-(R\A)}$. 
By Lemma \ref{lemma:1.7}, we get $\Ext^1_{\mathcal{D}^-(R\A)}(P^\bullet, \mathfrak{m}_RA^\bullet)=0$, and thus there exists a morphism $g: P^\bullet_R\to A^\bullet$ in $\mathcal{K}^-(R\A)$ such that $\tau_{A^\bullet}\circ g = (\tau_{A^\bullet})_\ast(g)= \varphi^{-1}\circ \begin{pmatrix} 0\\\textup{pr}\end{pmatrix}$, where $\textup{pr}$ is the composition of the natural morphism $\tau_{P^\bullet_R}: P^\bullet_R\to \k\hat{\otimes}_R P^\bullet_R$ in $\mathcal{C}^-(R\A)$ with the isomorphism $\pi_{R, P^\bullet}$. Consider the distinguished triangle in $\mathcal{K}^-(R\A)$
\begin{equation}\label{triag:C}
P^\bullet_R\xrightarrow{g} A^\bullet \xrightarrow{h} C^\bullet \to TP^\bullet_R,
\end{equation}
where $C^\bullet= C^\bullet_g$ is the mapping cone of $g$. Note in particular that all the terms of $C^\bullet$ are all abstractly free finitely generated $R\A$-modules. Tensoring (\ref{triag:C}) with $\k$ over $R$ yields a commutative diagram of triangles in $\mathcal{K}^-(R\A)$ as in (\ref{triagproyec}). Note in particular that by the definition of mapping cones, $\k\hat{\otimes}_RC^\bullet$ is isomorphic to a {\bf bounded} complex in $\mathcal{C}^-(\A)$ whose terms are all in $\A\textup{-GProj}$. By Lemma \ref{lemma:2.8}, we obtain that $g$ is a section in $\mathcal{K}^-(R\A)$. Following \cite[Lemma 1.4]{happel}, we obtain that $h:A^\bullet\to C^\bullet$ is a retraction in $\mathcal{K}^-(R\A)$, i.e., there exists a morphism $k: C^\bullet \to A^\bullet$ in $\mathcal{K}^-(R\A)$ such that $h\circ k=1_{C^\bullet}$.  By tensoring with $\k$ over $R$, we obtain a retraction $\bar{k}=\k\hat{\otimes}_Rk: \k\hat{\otimes}_RC^\bullet \to \k\hat{\otimes}_R A^\bullet$ of $\bar{h}$ in $\mathcal{K}^-(\A)$ . Consider the morphism $(k, g): C^\bullet\oplus P^\bullet_R\to T^\bullet$ in $\mathcal{K}^-(R\A)$. Let $\iota_{P^\bullet}: P^\bullet \to V^\bullet \oplus P^\bullet$ (resp. $p_{V^\bullet}: V^\bullet\oplus P^\bullet\to V^\bullet$) the natural morphism in $\mathcal{C}^-(\A)$ which is injective (resp. surjective) on terms. Since $\varphi\circ \bar{g}=\varphi\circ (\k\hat{\otimes}_Rg)\circ \pi_{R, P^\bullet}^{-1}=\iota_{P^\bullet}$, it follows from the axiom of triangulated categories (TR3) (see e.g. \cite[\S I.1]{hartshorne}) that there exists a morphism $\xi: \k\hat{\otimes}_RC^\bullet \to V^\bullet$ such that the following diagram of exact triangles in $\mathcal{K}^-(\A)$ is commutative:

\begin{equation*}
\begindc{\commdiag}[\Sized]
\obj(0,2)[p0]{$P^\bullet$}
\obj(4,2)[p1]{$\k\hat{\otimes}_RA^\bullet$}
\obj(8,2)[p2]{$\k\hat{\otimes}_RC^\bullet$}
\obj(12,2)[p3]{$TP^\bullet_R$}
\obj(0,0)[q0]{$P^\bullet$}
\obj(4,0)[q1]{$V^\bullet\oplus P^\bullet$}
\obj(8,0)[q2]{$V^\bullet$}
\obj(12,0)[q3]{$TP^\bullet$}
\mor{p0}{p1}{$\bar{g}$}
\mor{p1}{p2}{$\bar{h}$}
\mor{p2}{p3}{}
\mor{q0}{q1}{$\iota_{P^\bullet}$}
\mor{q1}{q2}{$p_{V^\bullet}$}
\mor{q2}{q3}{}
\mor{p0}{q0}{$=$}
\mor{p1}{q1}{$\varphi$}
\mor{p2}{q2}{$\xi$}
\mor{p3}{q3}{$=$}
\enddc
\end{equation*} 

Since $\varphi$ is assumed to be an isomorphism in $\mathcal{K}^-(\A)$, it follows from \cite[Prop. I.1.1 (c)]{hartshorne} that $\xi$ is also an isomorphism in $\mathcal{K}^-(\A)$. By letting $p_{P^\bullet}: V^\bullet\oplus P^\bullet\to P^\bullet$ be the natural morphism in $\mathcal{C}^-(\A)$ that is surjective on terms, we obtain
\begin{align}\label{eqns:2.14}
p_{P^\bullet}\circ\varphi\circ(\k\hat{\otimes}_Rg)&=p_{P^\bullet}\circ \iota_{P^\bullet}\circ \pi_{R,P^\bullet}=\pi_{R,P^\bullet},\\
p_{V^\bullet}\circ \varphi\circ (\k\hat{\otimes}_Rg)&=p_{V^\bullet}\circ \iota_{P^\bullet}\circ \pi_{R,P^\bullet}=0,\notag\\
p_{V^\bullet}\circ \varphi\circ (\k\hat{\otimes}_Rk)&=\xi\circ \bar{h}\circ \bar{k}=\xi,\text{ and }\notag\\
p_{P^\bullet}\circ \varphi\circ (\k\hat{\otimes}_Rk)&=p_{P^\bullet}\circ \varphi\circ \bar{k}\notag
\end{align} 
By Claim \ref{claim1}, there exists a morphism $\lambda: C^\bullet \to P^\bullet_R$ such that $\k\hat{\otimes}_R\lambda =\pi_{R, P^\bullet}^{-1}\circ p_{P^\bullet}\circ \varphi\circ\bar{k}$.
If $k_1=k-g\circ \lambda$, then $h\circ k_1=1_{C^\bullet}$ and $\bar{h}\circ (\k\hat{\otimes}_Rk_1)=1_{\k\hat{\otimes}_RC^\bullet}$. Hence, by replacing $k$ for $k_1$ in the identities (\ref{eqns:2.14}), we see that $p_{V^\bullet}\circ \varphi\circ (\k\hat{\otimes}_Rk_1)=\xi$ and that $p_{P^\bullet}\circ \varphi\circ(\k\hat{\otimes}_Rk_1)=p_{P^\bullet}\circ \varphi\circ \bar{k}-p_{P^\bullet}\circ \varphi\circ (\k\hat{\otimes}_Rg)\circ (\k\hat{\otimes}_R\lambda)=0$. It follows that $(k_1,g): C^\bullet \oplus P^\bullet_R\to A^\bullet$ provides an isomorphism in $\mathcal{K}^-(R\A)$ between the quasi-lifts $(C^\bullet \oplus P^\bullet_R, \xi\oplus \pi_{R, P^\bullet})$ and $(A^\bullet, \varphi)$ of $V^\bullet \oplus P^\bullet$ over $R$. Therefore, the map (\ref{mapdef}) is bijective for all Artinian objects $R$ in $\Ca$. This finishes the proof of Lemma \ref{lemma:2.10}.
\end{proof}

\section{Finitely Generated Gorenstein Projective Modules, Singular Equivalences of Morita Type, and Versal Deformation Rings}\label{section4}

Recall that finitely generated Gorenstein projective left $\A$-modules coincide with the left $\A$-modules of Gorenstein dimension zero or the ones that are totally reflexive as discussed in \cite{auslander4} and \cite{avramov2}, respectively (see e.g. \cite[Lemma 2.1.4]{chenxw4}).

We need the following result that summarizes some properties of finitely generated Gorenstein projective left $\A$-modules (see \cite[\S 4]{auslander3} and \cite{avramov2}).

\begin{lemma}\label{lemma:3.1}
Let $V$ be a finitely generated left $\A$-module. 
\begin{enumerate}
\item $V$ is Gorenstein projective with finite projective dimension if and only if $V$ is projective.
\item If $V$ is an object in $\A\textup{-Gproj}$, then for all $j\geq 1$, there exists a finitely generated left $\A$-module $W$ such that $V= \Omega^j W$, where $\Omega$ denotes the syzygy operator.
\end{enumerate}
\end{lemma} 

The following definition was introduced by X. W. Chen and L. G. Sun in \cite{chensun}, which was further studied by G. Zhou and A. 
Zimmermann in \cite{zhouzimm}, as a way of generalizing the concept of stable equivalence of Morita type introduced by M. Brou\'e in 
\cite{broue}.  
\begin{definition}\label{defi:3.2}
Let $\A$ and $\Gamma$ be finite-dimensional $\k$-algebras, and let $X$ be a $\Gamma$-$\A$-bimodule and $Y$ a $\A$-$\Gamma$-bimodule. We say that $X$ and $Y$ induce a {\it singular equivalence of Morita type} between $\A$ and $\Gamma$ (and that $\A$ and $
\Gamma$ are {\it singularly equivalent of Morita type}) if the following conditions are satisfied:
\begin{enumerate}
\item $X$ is finitely generated and projective as a left $\Gamma$-module and as a right $\A$-module.
\item $Y$ is finitely generated and projective as a left $\A$-module and as a right $\Gamma$-module. 
\item There is a finitely generated $\Gamma$-$\Gamma$-bimodule $Q$ with finite projective dimension such that $X\otimes_\A Y\cong 
\Gamma \oplus Q$ as $\Gamma$-$\Gamma$-bimodules.
\item There is a finitely generated $\A$-$\A$-bimodule $P$ with finite projective dimension such that $Y\hat{\otimes}_\Gamma X\cong \A\oplus P
$ as $\A$-$\A$-bimodules.
\end{enumerate}
\end{definition}
It follows from \cite[Prop. 2.3]{zhouzimm} that singular equivalences of Morita type induce equivalences of singularity categories.

\begin{remark}\label{rem:3.3}
The concept of singular equivalence of Morita type was further generalized by Z. Wang in \cite{wang}, where the concept of {\it singular 
equivalence of Morita type with level} is introduced. Moreover, in \cite[Prop. 2.6]{skart}, \O. Skarts{\ae}terhagen proved that if two algebras 
are singularly equivalent of Morita type, then they are also singularly equivalent of Morita type with level. Therefore, it follows from the steps 1-3 within the proof of 
\cite[Prop. 4.5]{wang} that if ${_\Gamma}X_\A$ and ${_\A}Y_\Gamma$ are bimodules which induce a singular equivalence of Morita 
type between two finite-dimensional $\k$-algebras $\A$ and $\Gamma$ as in Definition \ref{defi:3.2}, and such that $\Hom_\Gamma(X,\Gamma)$ and $\Hom_\A(Y, \A)$ are of finite projective dimension as a left $\A$-module and as a left $\Gamma$-module, respectively, then the functors 
\begin{align}\label{cmequiv0}
X\hat{\otimes}_\A-:\A\textup{-mod} \to \Gamma \textup{-mod}&& \text{ and } && Y\hat{\otimes}_\Gamma-:\Gamma\textup{-mod} \to \A\textup{-mod} 
\end{align}
send finitely generated Gorenstein projective modules to finitely generated Gorenstein projective modules.
\end{remark}

\begin{proposition}\label{propimp}
Let $\A$ and $\Gamma$ be finite dimensional $\k$-algebras, and assume that ${_\Gamma}X_\A$ and ${_\A} Y_\Gamma$ are bimodules that induce a singular 
equivalence of Morita type between $\A$ and $\Gamma$ as in Definition \ref{defi:3.2}, and such that $\Hom_\Gamma(X,\Gamma)$ and $\Hom_\A(Y, \A)$ are of finite projective dimension as a left $\A$-module and as a left $\Gamma$-module, respectively. If $V$ is a finitely generated Gorenstein projective left $\A$-module, then $X
\hat{\otimes}_\A V$ is a finitely generated Gorenstein projective left $\Gamma$-module. Moreover, if $P$ is a finitely generated $\A$-$\A$-bimodule with finite projective 
dimension satisfying Definition \ref{defi1} (iv), then $P\hat{\otimes}_\A V$ is a finitely generated projective left $\A$-module.    
\end{proposition}

\begin{proof}
Assume that $V$ is a finitely generated Gorenstein projective left $\A$-module. It follows from Remark \ref{rem:3.3} that $X\hat{\otimes}_\A V$ is a finitely generated Gorenstein projective left $\Gamma$-module. Assume next that $P$ is a finitely generated $\A$-$\A$-bimodule with finite projective dimension and which satisfies Definition 
\ref{defi:3.2} (iv).  Then there exists a finite projective resolution of the $\A$-$\A$-bimodule $P$:
\begin{equation}\label{projectiveP}
0\to P^s\xrightarrow{\delta^s}P^{s-1}\to\cdots\to P^1\xrightarrow{\delta^1}P^0\to 0,
\end{equation}
where $s\geq 0$ is the projective dimension of $P$ as a $\A$-$\A$-bimodule.  Since for every $0\leq j\leq s$, $P^j$ is a 
projective $\A$-$\A$-bimodule, it follows that $P^j$ is also a left and a right projective $\A$-module. In particular, the 
complex (\ref{projectiveP}) is a projective resolution of $P$ as a right $\A$-module, which implies that $P$ has finite projective dimension 
as a right $\A$-module. In particular, the right $\A$-module $\Omega^s_\A P=P^s$ is a projective right $\A$-module. Since $V$ is a Gorenstein projective, it follows from Lemma \ref{lemma:3.1} (ii) that there exists a finitely generated left $\A$-module $W_1$ such that $V=\Omega^s_\A W_1$. Using dimension shifting together with the above observations, we get that for all $k\geq 1$,
\begin{align*}
\mathrm{Tor}_k^\A(P,V)\cong \mathrm{Tor}_k^\A(P,\Omega^s_\A W_1)\cong \mathrm{Tor}_k^\A(\Omega^s_\A P,W_1)=0.  
\end{align*}
Thus, we get a long exact sequence of left $\A$-modules
\begin{equation}\label{projectivePV}
0\to P^s\hat{\otimes}_\A V\xrightarrow{\delta^s\hat{\otimes} 1_V}P^{s-1}\hat{\otimes}_\A V\to\cdots\to P^1\hat{\otimes}_\A V\xrightarrow{\delta^1\hat{\otimes} 1_V}
P^0\hat{\otimes}_\A V\to 0.
\end{equation}
Since for all $0\leq j\leq s$, the $\A$-$\A$-bimodule $P^j$ in (\ref{projectiveP}) is finitely generated and projective, it follows that $P^j
\hat{\otimes}_\A V$ is a finitely generated projective left $\A$-module, which implies that (\ref{projectivePV}) is a projective resolution of $P
\hat{\otimes}_\A V$. Therefore $P\hat{\otimes}_\A V$ has finite projective dimension as a left $\A$-module. 

On the other hand, since $P$ satisfies Definition \ref{defi1} (iv), it follows that 
\begin{align*}
Y\hat{\otimes}_\Gamma(X\hat{\otimes}_\A V)\cong V\oplus (P\hat{\otimes}_\A V) \text{ as left $\A$-modules}.
\end{align*}
By hypothesis and by Remark \ref{rem:3.3}, we have that both $V$ and  $Y\hat{\otimes}_\Gamma(X\hat{\otimes}_\A V)$ are both finitely generated Gorenstein projective left $\A$-modules, and therefore so is $P\hat{\otimes}_\A V$. It follows from Lemma \ref{lemma:3.1}(i) that $P\hat{\otimes}_\A V$ is a finitely generated projective left $\A$-module. This 
finishes the proof of Proposition \ref{propimp}. 
\end{proof}

We obtain the following immediate consequence of Proposition \ref{propimp} involving complexes.

\begin{corollary}\label{cor:3.5}
Let $\A$ and $\Gamma$ be finite dimensional $\k$-algebras, and assume that ${_\Gamma}X_\A$, ${_\A} Y_\Gamma$ and $P$ are as in Proposition \ref{propimp}. If $V^\bullet$ is a {\bf bounded} complex in $\mathcal{C}^-(\A)$ whose terms are all finitely generated Gorenstein projective left $\A$-modules, then $P\hat{\otimes}_\A V^\bullet$ is an object in $\mathcal{K}^b(\A\textup{-proj})$.
\end{corollary}

Before we state the second important result of this article, we should note that by \cite[Thm. 3.1.4]{blehervelez2}, versal deformation rings of bounded complexes of finitely generated modules over finite-dimensional algebras are preserved by derived equivalences induced by {\it nice two-sided tilting complexes} (as introduced in \cite[Def. 3.1.1]{blehervelez2}). However, by e.g. \cite{dugas} and its references, not every stable equivalence of Morita type between self-injective algebras is induced by a derived equivalence. Moreover, it is shown in \cite[\S 4]{bekkert-giraldo-velez} that equivalences of singularity categories not necessarily induce singular equivalences of Morita type.     

\begin{theorem}\label{thm:3.6}
Let $\A$ and $\Gamma$ be finite dimensional $\k$-algebras, and assume that ${_\Gamma}X_\A$, ${_\A} Y_\Gamma$ and $P$ are as in Proposition \ref{propimp}. Let $V^\bullet$ be a {\bf bounded} complex in $\mathcal{C}^-(\A)$ whose terms are all finitely generated Gorenstein projective left $\A$-modules. Then the terms of $X\hat{\otimes}_\A V^\bullet$ are all finitely generated Gorenstein projective left $\Gamma$-modules, and the versal deformation rings $R(\A, V^\bullet)$ and $R(\Gamma, X\hat{\otimes}_\A V^\bullet)$ are isomorphic in $\hat{\Ca}$. Moreover, $R(\A, V^\bullet)$ is universal if and only if $R(\Gamma, X\hat{\otimes}_\A V^\bullet)$ is universal.
\end{theorem}

Although the proof of Theorem \ref{thm:3.6} is easily obtained by adjusting the arguments in the proof of \cite[Prop. 3.2.6]{blehervelez2} to our context, we decided to include it for the convenience of the reader.

\begin{proof}[Proof of Theorem \ref{thm:3.6}]
Let $R$ be an Artinian object in $\Ca$. 
%By Remark \ref{rem:1.3}, we can assume as before that $V^\bullet$ is a bounded above complex of abstractly free finitely generated $\A$-modules.
Note that $X_R=R\hat{\otimes}_\k X$ (resp. $Y_R=R\hat{\otimes}_\k Y$) is projective as left $R\Gamma$-module (resp. $R\A$-module) and as right $R\A$-module (resp. $R\Gamma$-module). Since $X_R\hat{\otimes}_{R\A}Y_R\cong R\hat{\otimes}_\k(X\hat{\otimes}_\A Y)$, following Definition \ref{defi:3.2} we obtain
\begin{align*}
Y_R\hat{\otimes}_{R\Gamma} X_R&\cong R\A\oplus P_R &&\text{ as $R\A$-$R\A$-bimodules, and }\\
X_R\hat{\otimes}_{R\A} Y_R&\cong R\Gamma\oplus Q_R &&\text{ as $R\Gamma$-$R\Gamma$-bimodules,}
\end{align*}
where $P_R=R\hat{\otimes}_\k P$ (resp. $Q_R=R\hat{\otimes}_k Q$) is a $R\A$-$R\A$-bimodule (resp. $R\Gamma$-$R\Gamma$-bimodule) of finite projective dimension. By Corollary \ref{cor:3.5} we have that $P\hat{\otimes}_\A V^\bullet$ is an object in $\mathcal{K}^b(\A\textup{-proj})$, and thus by  Lemma \ref{lemma:2.10}, we have that $P\hat{\otimes}_\A V^\bullet$ has a universal deformation ring $R(\A, P\hat{\otimes}_\A V^\bullet)$ isomorphic to $\k$. In particular, every quasi-lift of $P\hat{\otimes}_\A V^\bullet$ over $R$ is isomorphic to the trivial quasi-lift $(R\hat{\otimes}_\k(P\hat{\otimes}_\A V^\bullet), \pi_{R, P\hat{\otimes}_\A V^\bullet})$ of $P\hat{\otimes}_\A V^\bullet$ over $R$, where $\pi_{R, P\hat{\otimes}_\A V^\bullet}: \k\hat{\otimes}_R(R\hat{\otimes}_\k(P\hat{\otimes}_\A V^\bullet))\to P\hat{\otimes}_\A V^\bullet$ is the natural isomorphism in $\mathcal{K}^-(\A)$. 
Let now $(M^\bullet, \phi)$ be a quasi-lift of $V^\bullet$ over $R$. Since $R$ is Artinian, by Remark \ref{rem:1.3} we can assume as before that $V^\bullet$ (resp. $M^\bullet$) is a bounded above complex of abstractly free finitely generated $\A$-modules (resp. $R\A$-modules), and that $\phi$ is given by a quasi-isomorphism in $\mathcal{C}^-(\A)$ that is surjective on terms. Let $M'^\bullet=X_R\hat{\otimes}_\A M^\bullet$. Note that all the terms of $M'^\bullet$ are finitely generated left $R\Gamma$-modules. Since $X_R$ is a finitely generated projective right $R\A$-module and since $M^\bullet$ has terms that are all free over $R\A$, it follows that all the terms of $M'^\bullet$ are abstractly free finitely generated $R\Gamma$-modules. We next observe that we can view $\A$-mod as the full subcategory of $R\A$-mod that consists in all those objects in $R\A$-mod on which the action of the maximal ideal $\m_R$ of $R$ is trivial. Moreover, on $\A$-mod the functor $X_R\hat{\otimes}_\A-$ coincides with the functor $X\hat{\otimes}_\A-$. Define $\phi'=X_R\hat{\otimes}_{R\A}\phi$. Since $\phi$ is assumed to be a quasi-isomorphism in $\mathcal{C}^-(\A)$ and $X_R$ is projective as a right $R\A$-module, it follows that $\phi'$ is also a quasi-isomorphism in $\mathcal{C}^-(\Gamma)$. Moreover,
\begin{equation}\label{eqn:3.4}
M'^\bullet\hat{\otimes}_R\k=(X_R\hat{\otimes}_{R\A}M^\bullet)\hat{\otimes}_R\k=X_R\hat{\otimes}_{R\A}(M^\bullet\hat{\otimes}_R\k)\xrightarrow{\phi'}X_R\hat{\otimes}_{R\A}V^\bullet=V'^\bullet,
\end{equation}
which means $(M'^\bullet, \phi')$ is a quasi-lift of $V'^\bullet$ over $R$. We therefore obtain for all Artinian objects $R$ in $\Ca$, a well-defined map
\begin{align*}
\Xi_R:\Def_\A(V^\bullet, R)&\to\Def_\A(V'^\bullet, R)\\
[M^\bullet, \phi]&\mapsto [M',\phi]=[X_R\hat{\otimes}_{R\A}M^\bullet, X_R\hat{\otimes}_{R\A}\phi]. 
\end{align*}
We need to show that $\Xi_R$ is bijective. Arguing as in (\ref{eqn:3.4}), we see that $(Y_R\hat{\otimes}_{R\Gamma}M'^\bullet,Y_R\hat{\otimes}_{R\Gamma}\phi')$
is a quasi-lift of $Y\hat{\otimes}_\Gamma V'^\bullet\cong V^\bullet \oplus (P\hat{\otimes}_\Lambda V^\bullet)$ over $R$. Moreover, 
\begin{eqnarray}
\label{eq:olala}
(Y_R\hat{\otimes}_{R\Gamma}M'^\bullet,Y_R\hat{\otimes}_{R\Gamma}\phi')
&\cong&((R\Lambda\oplus P_R)\hat{\otimes}_{R\Lambda}M^\bullet, (R\Lambda\oplus P_R)\hat{\otimes}_{R\Lambda}\phi) \\
&\cong&(M^\bullet\oplus (P_R\hat{\otimes}_{R\Lambda}M^\bullet), \phi\oplus(P_R\hat{\otimes}_{R\Lambda}\phi)). \nonumber
\end{eqnarray}
Since $(P_R\hat{\otimes}_{R\Lambda}M^\bullet,P_R\hat{\otimes}_{R\Lambda}\phi)$ is
a quasi-lift of $P\hat{\otimes}_\Lambda V^\bullet$ over $R$, which is a perfect complex over $\A$,
it follows from Lemma \ref{lemma:2.10} that $\Xi_R$ is injective.

Now let $(L^\bullet,\psi)$ be a quasi-lift of $V'=X\hat{\otimes}_{\Lambda}V^\bullet$ over $R$. As before, since $R$ is Artinian, we can assume that $L^\bullet$ is a bounded above complex of abstractly free finitely generated $R\Gamma$-modules, and that $\psi$ is given by a quasi-isomorphism in $\mathcal{C}^-(\Gamma)$ that is surjective on terms. Then 
$(L'^\bullet,\psi')= (Y_R\hat{\otimes}_{R\Gamma}L^\bullet,Y_R\hat{\otimes}_{R\Gamma}\psi)$ is a quasi-lift of 
$V''^\bullet=Y\hat{\otimes}_{\Gamma}V'^\bullet\cong V^\bullet \oplus (P\hat{\otimes}_\Lambda V^\bullet)$ over $R$. By Lemma \ref{lemma:2.10}, there exists a quasi-lift 
$(M^\bullet,\phi)$ of $V^\bullet$ over $R$ such that $(L'^\bullet,\psi')$ is isomorphic to the quasi-lift
$\left(M^\bullet\oplus (R\hat{\otimes}_k (P\hat{\otimes}_\Lambda V^\bullet)),\phi\oplus\pi_{R,P\hat{\otimes}_\Lambda V^\bullet}\right)$ of $P\hat{\otimes}_RV^\bullet$ over $R$. In particular, we can assume that the terms of $M^\bullet$ are abstractly free finitely generated $R\A$-modules, and $\phi$ is given by a quasi-isomorphism in $\mathcal{C}^-(\A)$ that is surjective on terms.
Arguing similarly as in (\ref{eq:olala}), we then have that $(L'^\bullet,\psi')$ is isomorphic to $(M''^\bullet,\phi'')=
(Y_R\hat{\otimes}_{R\Gamma}M'^\bullet,Y_R\hat{\otimes}_{R\Gamma}\phi')$ where $(M'^\bullet,\phi')=(X_R \hat{\otimes}_{R\Lambda}M^\bullet,X_R\hat{\otimes}_{R\Lambda}\phi)$. Therefore, $(X_R\hat{\otimes}_{R\Lambda}L'^\bullet ,X_R\hat{\otimes}_{R\Lambda}\psi')\cong (X_R\hat{\otimes}_{R\Lambda}M''^\bullet,X_R\hat{\otimes}_{R\Lambda}\phi'')$.
Arguing again similarly as in (\ref{eq:olala}) and using Lemma \ref{lemma:2.10}, we have 
\begin{eqnarray*}
(X_R\hat{\otimes}_{R\Lambda}L'^\bullet ,X_R\hat{\otimes}_{R\Lambda}\psi') 
&\cong&\left(L^\bullet\oplus (R\hat{\otimes}_k (Q\hat{\otimes}_\Gamma V'^\bullet)),\psi\oplus \pi_{R,Q\hat{\otimes}_\Gamma V'^\bullet}\right),
\mbox{ and }\\
(X_R\hat{\otimes}_{R\Lambda}M''^\bullet,X_R\hat{\otimes}_{R\Lambda}\phi'')
&\cong&\left(M'^\bullet \oplus (R\hat{\otimes}_k (Q\hat{\otimes}_\Gamma V'^\bullet)),\phi'\oplus \pi_{R,Q\hat{\otimes}_\Gamma V'^\bullet}\right).
\end{eqnarray*}
Thus by Lemma \ref{lemma:2.10}, it follows that $(L^\bullet,\psi)\cong (M'^\bullet,\phi')$, i.e. $\Xi_R$ is surjective.

To show that the maps $\Xi_R$ are natural with respect to morphisms $\alpha:R\to R'$ in 
$\mathcal{C}$, consider $(M^\bullet,\phi)$ and $(M'^\bullet,\phi')$ as above.
Since $X_R$ is a projective right $R\Lambda$-module and the terms of 
$M^\bullet$ can be assumed to be abstractly free finitely generated $R\A$-modules, there exists a natural isomorphism in $\mathcal{K}^-(R\A)$:
\begin{equation*}
f\;:\quad R' \hat{\otimes}_{R,\alpha}M'^\bullet= 
R'\hat{\otimes}_{R,\alpha}(X_R\hat{\otimes}_{R\Lambda}M^\bullet)
\to X_{R'}\hat{\otimes}_{R'\Lambda}(R'\hat{\otimes}_{R,\alpha} M^\bullet).
\end{equation*}
It is straightforward to see that $f$ provides an isomorphism between the quasi-lifts $(R'\hat{\otimes}_{R,\alpha}M'^\bullet,(\phi')_\alpha)$ and 
$(X_{R'} \hat{\otimes}_{R'\Lambda} (R'\hat{\otimes}_{R,\alpha}M^\bullet),X_{R'}\hat{\otimes}_{R'\Lambda}(\phi_\alpha))$
of $V'^\bullet$ over $R'$, where $\phi_\alpha$ (resp. $(\phi')_\alpha$) is the composition of quasi-isomorphisms $R'\hat{\otimes}_{R,\alpha}((X_R\hat{\otimes}_{R\Lambda}M^\bullet)\hat{\otimes}_R\k)\cong M'^\bullet\hat{\otimes}_R\k\xrightarrow{\phi} V^\bullet$ (resp. $X_{R'}\hat{\otimes}_{R'\Lambda}(R'\hat{\otimes}_{R,\alpha} M^\bullet)\hat{\otimes}_{R'}\k\cong (X_{R'}\hat{\otimes}_{R'\Lambda}M^\bullet)\hat{\otimes}_{R}\k\xrightarrow{\phi'} V^\bullet$).    
Since the deformation functors $\hat{\Fun}_{V^\bullet}$ and $\hat{\Fun}_{V'^\bullet}$ are continuous, this implies that they are naturally isomorphic. Hence the versal deformation rings $R(\Lambda,V^\bullet)$ and $R(\Gamma,V'^\bullet)$ are isomorphic in $\hat{\mathcal{C}}$.
\end{proof}

Since singular equivalences of Morita type induce equivalences of singularity categories as noted before, we obtain the following consequence of Theorem \ref{thm:3.6}.

\begin{corollary}
Let $\A$ and $\Gamma$ be finite dimensional $\k$-algebras, and assume that ${_\Gamma}X_\A$, ${_\A} Y_\Gamma$ and $P$ are as in Proposition \ref{propimp}. Let $V^\bullet$ be a {\bf bounded} complex in $\mathcal{C}^-(\A)$ whose terms are all finitely generated Gorenstein projective $\A$-modules, and such that $\Hom_{\mathcal{D}_\textup{sg}(\A\textup{-mod})}(V^\bullet, V^\bullet)=\k$. Then the terms of $X\hat{\otimes}_\A V^\bullet$ are all finitely generated Gorenstein projective $\Gamma$-modules, and 
\begin{equation*}
\Hom_{\mathcal{D}_\textup{sg}(\Gamma\textup{-mod})}(X\hat{\otimes}_\A V^\bullet, X\hat{\otimes}_\A V^\bullet)=\k. 
\end{equation*}
Moreover, the universal deformation rings $R(\A, V^\bullet)$ and $R(\Gamma, X\hat{\otimes}_\A V^\bullet)$ are isomorphic in $\hat{\Ca}$. 
\end{corollary}
\appendix
\section{A Remark Concerning Universal Deformation Rings of Finitely Generated Gorenstein Projective $\A$-modules over Finite Dimensional Algebras}

As before, we assume that $\k$ is a field of arbitrary characteristic and that $\A$ is an arbitrary but fixed finite dimensional $\k$-algebra.

By using the same arguments as those in the proofs of \cite[Thm. 1.2]{velez2} and \cite[Thm. 3.4]{bekkert-giraldo-velez} together with Lemma \ref{lemma:3.1}, Proposition \ref{propimp}, the observations in the paragraph after the proof of \cite[Lemma 2.1]{chenxw3}, and \cite[Lemma 2.2]{chenxw3}, we can drop the conditions of $\A$ being a Gorenstein algebra and $V$ being a (maximal) Cohen-Macaulay left $\A$-module to obtain the proof of the following version of \cite[Thm. 1.2]{velez2} and \cite[Thm. 3.4]{bekkert-giraldo-velez}.

\begin{theorem}\label{thm:4.5}
Let $V$ be a finitely generated Gorenstein projective left $\A$-module.
\begin{enumerate}
\item If the stable endomorphism ring of $V$ is isomorphic to $\k$, then the versal deformation ring $R(\A,V)$ is universal. Moreover, $\Omega V$ also has stable endomorphism ring isomorphic to $\k$ and the universal deformation rings $R(\A, V)$ and $R(\A,\Omega V)$ are isomorphic.
\item Assume that $\Gamma$ is another finite dimensional $\k$-algebra such that there exist bimodules ${_\Gamma}X_\A$ and ${_\A} Y_\Gamma$ inducing a singular 
equivalence of Morita type between $\A$ and $\Gamma$ as in Definition \ref{defi:3.2} and which satisfy the hypothesis of Proposition \ref{propimp}. Then $X\hat{\otimes}_\A V$ is a finitely generated Gorenstein projective left $\Gamma$-module, and the versal deformation rings $R(\A, V)$ and $R(\Gamma, X\hat{\otimes}_\A V)$ are isomorphic. In particular, $R(\A,V)$ is universal if and only if $R(\Gamma, X\hat{\otimes}_\A V)$ is universal.
\end{enumerate} 
\end{theorem}

We now apply the results in Theorem \ref{thm:4.5} (i) to a particular non-Gorenstein finite dimensional $\k$-algebra. 
\begin{example}
Let $\A_0$ be the basic $3$-cycle Nakayama $\k$-algebra with quiver $Z_3$ as in Figure \ref{fig1} with normalized admissible sequence $\mathbf{c}(\A_0)=(8,9,9)$ as explained in \cite[\S 3.1]{chenye}. 
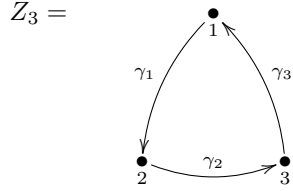
\begin{figure}[ht]
\begin{align*}
Z_3 &=\xymatrix@1@=20pt{
&&\underset{1}{\bullet}\ar@/_/[ddl]_{\gamma_1}&&\\\\
&\underset{2}{\bullet}\ar@/_/[rr]^{\gamma_2}&&\underset{3}{\bullet}\ar@/_/[uul]_{\gamma_3}&
}
%\\\\
%I&= \{\gamma_1(\gamma_3\gamma_2\gamma_1)^3, \gamma_3\gamma_2(\gamma_1\gamma_3\gamma_2)^3, \gamma_1\gamma_3(\gamma_2\gamma_1\gamma_3)^3\}.
\end{align*}
\caption{The quiver of the basic $3$-cycle Nakayama $\k$-algebra $\A_0$.}\label{fig1}
\end{figure}

It follows from \cite[Prop. 3.14 (3)]{chenye} for the case $k=2$ that $\A_0$ is a non-Gorenstein Nakayama algebra which is not CM-free in the sense of \cite{chenxw3}. Moreover, the isomorphism classes of finitely generated indecomposable Gorenstein projective $\A_0$-modules are represented by the string $\A_0$-modules (in the sense of \cite{buri}) $V_1=M[\gamma_3\gamma_2]$ and $V_2=M[\gamma_3\gamma_2(\gamma_1\gamma_3\gamma_2)]$. Since $\A_0$ is in particular special biserial, we can use the description of morphisms between string $\A_0$-modules as discussed in \cite{krause}. It is straightforward to show that $\End_{\A_0}(V_1)=\k$ and $\SEnd_{\A_0}(V_2)=\k$. Moreover, using the composition series of the indecomposable projective $\A_0$-modules, we obtain that $\Omega V_2=V_1$. We also obtain that $\Ext_{\A_0}^1(V_2,V_2)=\k$, which implies that the universal deformation ring $R(\A,V_2)$ of $V_2$ is isomorphic to a quotient of $\k[[t]]$. In view of Theorem \ref{thm:4.5}(i), we calculate $R(\A_0,V_2)$ by calculating $R(\A,V_1)$. Note that by e.g. \cite[Lemma 2.1.8]{chenxw4}, we have $\Ext_{\A_0}^1(V_1,V_1)=\SHom_\A(\Omega V_1, V_1)=\k$. Let $M_1=M[\gamma_3\gamma_2\gamma_1\gamma_3\gamma_2]$ and consider the following short exact sequences of $\A_0$-modules:
\begin{align}
0\to V_1\xrightarrow{\iota_1} M_1\xrightarrow{\pi_1}V_1\to 0,\label{seqn1}\\
0\to M_1\xrightarrow{\iota_2} P_2\xrightarrow{\pi_2}M_1\to 0,\label{seqn2}
\end{align}
where $P_2$ is the projective cover of  the simple $\A_0$-module $S_2$ corresponding to the vertex $2$. We obtain then by using (\ref{seqn1}) and (\ref{seqn2}) that $M_1$ and $P_2$ define lifts in the sense of \cite{blehervelez} of $V_1$ over $\k[[t]]/(t^2)$ and $\k[[t]]/(t^3)$,  respectively. In particular, $M_1$ defines a non-trivial lift of $V_1$ over the ring of dual numbers $\k[[t]]/(t^2)$, where the action of $t$ is given by $\iota_1 \circ \pi_1$. This implies that there exists a unique surjective $\k$-algebra morphism $\psi: R(\A_0, V_1])\to \k[[t]]/(t^2)$ in $\hat{\Ca}$ corresponding to the deformation defined by $M_1$. By using similar arguments to those in the proof of \cite[Claim 4.1.1]{velez}, it is straightforward to show that $R(\A,V_2)\cong R(\A,V_1)\cong \k[[t]]/(t^3)$. 
\end{example}

%\nocite{*}
\bibliographystyle{amsplain}
\bibliography{UniversalComplexesGorenstein}

\end{document}